\documentclass[12pt]{article}
\usepackage{graphicx}
\usepackage{color}
\usepackage{amsmath,amsthm}
\usepackage{amssymb,enumerate}
\theoremstyle{plain}% default
\newtheorem{thm}{Theorem}[section]
\newtheorem{lemma}[thm]{Lemma}
\newtheorem{prop}[thm]{Proposition}
\newtheorem{cor}[thm]{Corollary}

\theoremstyle{definition}
\newtheorem{defn}[thm]{Definition}
\theoremstyle{remark}

\newcommand{\N}{{\rm I\kern-.22em N}}
\newcommand{\Z}{{\sf Z\kern-.42em Z}}
\newcommand{\R}{\mathbb{R}}
\newcommand{\BbbC}
  {{\rm\kern.22em\rule[.1ex]{.06em}
       {1.4ex}\kern-.28em C}}
\newcommand{\BbbQ}
  {{\rm\kern.22em\rule[.1ex]{.06em}
       {1.4ex}\kern-.28em Q}}

\makeatother

\author{
Jos\'{e} A. {\sc Carrillo}$^{\ast}$ and
Yoshie {\sc Sugiyama}$\dag$
\vspace{3mm} \\
{\small $^{\ast}$
Department of Mathematics,} \\
{\small Imperial College London,}
{\small SW7 2AZ London, {\sc United Kingdom},}
\vspace{2mm} \\
{\small
$\dag$ Faculty of Mathematics,} \\
{\small  Kyushu University, Fukuoka 819-0395, {\sc Japan}}
}

\title{Compactly supported stationary states of the degenerate Keller-Segel
system in the diffusion-dominated regime}

\begin{document}

\maketitle
 \vskip5mm
\begin{abstract}
We first show the existence of unique global minimizer of the free energy for all masses associated to a nonlinear diffusion version of the classical Keller-Segel model when the diffusion dominates over the attractive force of the chemoattractant. The strategy uses an approximation of the variational problem in the whole space by the minimization problem posed on bounded balls with large radii. We show that all stationary states in a wide class coincide up to translations with the unique compactly supported radially decreasing and smooth inside its support global minimizer of the free energy. Our results complement and show alternative proofs with respect to \cite{Strohmer,LY,Kim-Yao}.
\end{abstract}
\maketitle
\section{Introduction}

In this work, we are interested in obtaining a full self-contained characterization of the stationary states in a certain class of functions for the following nonlinear version of the Keller-Segel (KS) model for chemotaxis
$$
\left\{
    \begin{array}{llll}
         & \rho_t  =
             \nabla \cdot \Big( \nabla \rho^m - \rho \nabla c \Big),
                    & x \in \R^N, \ 0<t<\infty, \nonumber \\
         & - \Delta c = \rho,
                    & x \in \R^N, \ 0<t<\infty, \nonumber \\
         & \rho(x,0)   = \rho_0(x),
                    & x \in \R^N,
     \end{array}
  \right.
$$
where $N \ge 3$ and $m > 2-\frac{2}{N}$. The initial data $\rho_0$ is a non-negative function
in  $L^{1} \cap L^{\infty}(\R^N)$ with $\rho_0^m \in H^1(\R^N)$. The second equation can be solved by using the fundamental solution $\Gamma$ of the $-\Delta$ operator given by
\begin{eqnarray*} %\label{def-Gamma}
\Gamma(x)= \displaystyle \frac{1}{(N-2) \omega_{N-1} |x|^{N-2}}
\end{eqnarray*}
with $\omega_{N-1}$ denoting the area of the unit sphere $S^{N-1}$ in $\R^N$. Therefore, the previous system can be reduced to an aggregation-diffusion equation, as in \cite{CaMcCVi03,BlaCaCa07} for instance, where the nonlinear diffusion models the repulsion between cells/particles and the mean field chemotactic force $c=\Gamma\ast \rho$ models the nonlocal attraction via Newtonian interactions. This nonlinear diffusion version of the Keller-Segel model has been proposed as a remedy to take into account volume/size effects for the cells/particles, see \cite{HiPai01,PaiHi02,Kowalczyk04,CaCa06,LSV} and the references therein.

The aggregation-diffusion equation obtained by substituting $c$ onto the first equation in (KS) has a very nice variational structure. In fact, it is a gradient flow with respect to probability measures as recognized in many previous works and used effectively for understanding qualitative properties of solutions to aggregation-diffusion models, see \cite{CaMcCVi03,BlaCaCa07,BCL,BCC12,Bian-Liu,CLM14,LW} for instance. The free energy functional $E[\rho]$ 
\begin{eqnarray}
E[\rho] & := & \frac{1}{m-1} \int_{\R^N} \rho^m(x) \ dx
                  -  \frac{1}{2}\int_{\R^N} \int_{\R^N} \Gamma(x-y) \rho(x)\rho(y) \ dxdy.
\label{def-EU}
\end{eqnarray}
is formally a Liapunov functional for the evolution (KS) since formal variations fixing the mass of the density $\rho$ leads to
$$
\partial_t\rho(t)=\nabla \cdot \left( \rho(t)\, \nabla \frac{\delta E}{\delta \rho}[\rho(t)]\right)\qquad \mbox{ with } \frac{\delta E}{\delta \rho}[\rho](x) = \frac{m}{m-1}\rho^{m-1}(x)- \Gamma(x) \ast \rho(x)\,,
$$
which is equivalent to the (KS) system. Therefore, in order to find stable stationary states for our (KS) system, we should look for them among (local) minimizers of the free energy functional \eqref{def-EU}.

Concerning the assumptions on the nonlinear diffusion $m$ with respect to the dimension $N$, early works in \cite{KS-SuKu,Sblow-uq-2005,S-ks1,Sugi-threshold} studied the well-posedness in different cases when $m>1$ showing that the exponent $m=2-\frac{2}{N}$ plays an important role regarding global existence versus finite time blow-up of solutions. More precisely, it was proved that
\begin{itemize}
\item[(I)] Case $1 \le m \le 2-\frac{2}{N}$: the problem (KS) is solvable globally in time
for small initial data in $L^{\frac{N(2-m)}{2}}(\R^N)$.
\item[(II)] Case $1 \le m \le 2-\frac{2}{N}$: the problem (KS) can lead to a finite time blow-up for some large initial data.
\item[(III)] Case $m > 2-\frac{2}{N}$: the problem (KS) is solvable globally in time without any restriction on the size of the initial data.
\end{itemize}
This dichotomy of behaviors was clarified in \cite{BCL} where the authors studied the case $m=2-\frac{2}{N}$ and showed the existence of a critical parameter depending only on the mass and dimension $N$ as in the classical Keller-Segel model in two dimensions \cite{DoPe04,BlaDoPe06}. This sharp value is connected to the sharp constant of a variant of the classical HLS inequality which we will also use below. Moreover, recent works \cite{CCH1,CCH2} have identified the different scaling properties of the two competing effects of the free energy functional as the main reason behind this behavior dichotomy. In fact, if $m>2-\frac{2}{N}$ the nonlinear diffusion term in the functional dominates over the attractive part and one should expect the existence of asymptotically stable stationary states. From this point of view, a more updated revision of the dichotomy behaviors allow us to identify the following cases:
\begin{itemize}
\item[(I)] Aggregation-Dominated Case: $1 < m < 2-\frac{2}{N}$: global existence for some initial data and finite time blow-up coexist. There are can be unstable stationary states for some values of $m$, see \cite{CoPeZa04,Sblow-uq-2005,Luck-Sugi,KS-SuKu,S-ks1,Luck-Sugi-2,Sugi-threshold,CLW,BL,CW,CaCo14} for more information.
\item[(II)] Fair-Competition Case: $m = 2-\frac{2}{N}$: the problem (KS) exhibits a critical mass, see \cite{BCL} for the degenerate case. We refer to \cite{CCH2} for a comprehensive list of references in the classical linear diffusion case.
\item[(III)] Diffusion-dominated Case: $m > 2-\frac{2}{N}$: solutions exist globally in time without any restriction on the size of the initial data and they are uniformly bounded in time \cite{CaCa06,S-ks1,KS-SuKu}. Existence of localized radially symmetric compactly supported steady states have been proven in two dimensions for $m>1$ in \cite{CCV}. They have recently been shown to attract all solutions of (KS) for $m>1$ and $N=2$ for large times in \cite{CHVY}.
\end{itemize}

Our work can be considered a continuation in the effort of understanding the diffusion-dominated regime in the case of Newtonian interaction in any dimension. Early works \cite{LY,Strohmer}, see also \cite{Kim-Yao} and the references therein, showed the existence and uniqueness of compactly supported radial minimizers of the free energy \eqref{def-EU} in the three dimensional case for $m>\frac43$. We will give a self-contained proof of this result together with regularity and qualitative properties of the global minimizers for all dimensions $N\geq 3$, $m>2-\frac{2}{N}$. We will also characterize them in terms of solutions to an obstacle problem. Here, we follow the strategy in \cite{Strohmer} where the existence of unique radial minimizers of the free energy \eqref{def-EU} in the three dimensional case was established.

Let us state the main result in this work about stationary solutions of the degenerate Keller-Segel problem (KS). We start by defining properly the concept of stationary solution.

\begin{defn}\label{stationarystates}
Let $U_s \in L^1_+(\mathbb{R}^N)\cap L^\infty(\mathbb{R}^N)$ be a given function.
We call it {\it a stationary solution} of the (KS) system if $U_s^{m}\in H^1_{loc} (\mathbb{R}^N)$, $\nabla V_s:=\nabla \Gamma \ast U_s\in L^1_{loc} (\mathbb{R}^N)$, and it satisfies
\begin{equation}\label{steady}
\nabla U_s^{m} = - U_s \nabla V_s \quad \text{ in } \mathbb{\R}^N
\end{equation}
in the sense of distributions in $\R^N$.
\end{defn}

Let us point out that the recent work \cite{CHVY} shows that if stationary solutions of (KS) with $m>1$ exist in the sense of Definition \ref{stationarystates}, then they must be radially symmetric and decreasing about their center of mass. This result generalizes for a very large class of aggregation-diffusion equations, some symmetry results in \cite{Strohmer,CCV} obtained by moving plane techniques. We will make use of this radial symmetry result to identify all stationary states of the (KS) system as the global minimizers of the free energy except translations.

In what follows, we abbreviate simply as $\| \cdot \|_{r}$ the norm in $L^r(\R^N)$ and  $B_R$ denotes the standard open ball in $\R^N$ centered at the origin with the radius $R>0$. Our main result reads:

\begin{thm} \label{thm:ch-statinary-solution}
Let $N \ge 3$, $m > 2-\frac{2}{N}$, and $M>0$.
There exists a pair of functions $(U_M,V_M)$ with the following properties:
\begin{enumerate}
\item[(i)] $U_M$ is the unique radial global minimizer of the free energy $E[U]$ in the class
\begin{equation*}%\label{def-WM}
\mathcal{Y}_M := \{U \in L^1 \cap L^m(\R^N);
      \ \|U\|_1 = M, \ U(x) \ge 0 \quad
{\rm for \ a.e.} \   x \in \R^N\}\,.
\end{equation*}
with zero center of mass. Furthermore, it is radially decreasing and compactly supported on $B_{R_M}$, for some $R_M>0$. Moreover, all other global minimizers are determined by translations of $U_{M}$.

\item[(ii)] $U_{M}$ is a stationary solution of (KS) in the sense of Definition {\rm\ref{stationarystates}}.
In addition, $(U_{M},V_{M})$ satisfies \eqref{steady} in the classical sense in the interior of its support satisfying $U_{M} \in C_0(\R^N)\cap C^1(B_{R_M})$, $U_{M}^{m-1} \in W^{1,\infty}(\R^N)\cap W^{2,p}(B_{R_M})$ for all $1<p<\infty$, and $V_{M} \in C^2 \cap L^p(\R^N)$ \ for all $\frac{N}{N-2} < p \le \infty$.

\item[(iii)] Moreover, $U_{M}^{m-1+\delta} \in C^1(\R^N)$
for all $\delta>0$ with  $\nabla U_{M}^{m-1+\delta}(x)=0$ at $\partial B_{R_M}$.

\item[(iv)] The stationary solution satisfies
\begin{eqnarray*}
\frac{m}{m-1} U_{M}^{m-1}(x)
\ = \ \left( V_{M}(x) + \frac1M \int_{\R^N} \left(\frac{m}{m-1} U_{M}^{m-1} -V_M\right) U_M\, dx \right)_+ \,
\end{eqnarray*}
for all $x \in \R^N$.

\item[(v)] If $\tilde{U}_{M}$ is another stationary solution of {\rm (KS)} in the sense of Definition {\rm\ref{stationarystates}}, then $\tilde{U}_{M}$ is a translation of $U_{M}$.
\end{enumerate}
\end{thm}

As we already mentioned, the confinement for solutions of the (KS) system was recently proved in \cite{CHVY} in two dimensions for all $m>1$. However, in the case of $m>2-\frac{2}{N}$, $N\geq 3$, this confinement result is lacking and despite of the fact that there exists a global
solution $(\rho(t),c(t))$ of the (KS) system for an arbitrary large initial data $\rho_0$,
it is an open question to determine its asymptotic profile as $t \to \infty$. Let us point out that compactness of time diverging sequences is possibly not difficult to achieve using the methods in \cite{Sugi-threshold,Bian-Liu,CHVY}, but uniform in time moment control is difficult to show.
Therefore, we cannot prove or disprove that convergence towards the stationary state can coexist with dispersion of the mass to infinity for some choice of $m$, i.e., dichotomy of asymptotic behavior depending on the initial data cannot be excluded: solutions can converge to the localized stationary state or they can disperse as the Barenblatt solutions for the porous medium equation \cite{Ba1}.

This paper is organized as follows: Section \ref{section:Lane-Emden} is devoted to the
construction of a solution to the so-called Lane-Emden equation. Moreover, we prove its uniqueness and the compactness of the support of the solutions of the Lane-Emden equation.
In Section \ref{section:existence-stationary-point}, we first show the existence, uniqueness up to translations and regularity of global minimizers of the free energy. We analyse the minimization problem in the whole space by approximating it with a corresponding problem in bounded balls. We establish a uniform bound of the support of zero center of mass minimizers with respect to the radius of the ball using the properties of the Lane-Emden equation. We finally prove our main theorem by using the radial symmetry result in \cite{CHVY} to characterize all stationary states.

%%%%%%%%%%%%%%%%%%%%%%%%%%%%%%%%%%%%%%%%%%%%%%%%%%%%%%%%%%%%%%%%%%%%%%%%%%%%%%%%%%%%%%%%%%

\section{Lane-Emden equation}
\label{section:Lane-Emden}
\setcounter{equation}{0}

We generalize the construction of a solution of the so-called
Lane-Emden equation to $N\geq 3$ from the existing results for
$N=3$ in \cite{Chandra,S-Z,BU,Batt-P}, see also \cite{JL}. As remarked in
the introduction, we give a self-contained proof of the existence
of radial stationary solutions by dynamical system arguments. We also show its uniqueness and the property of compactness of the support of solutions based on them without resorting to nonlinear elliptic equations theory \cite{Bian-Liu,CL} nor to variational arguments as in \cite{Kim-Yao, LY}.

\begin{lemma}\label{lemma:local-existence}
Let $N \ge 3$ and $m > 2-\frac{2}{N}$. For every $\alpha>0$, there
exist a unique $0<R_*=R_*(\alpha)\leq\infty$ and a unique solution
$\psi \in C^1([0,R^*))\cap C^2((0,R_*))$ of
\begin{eqnarray}
\left\{
\begin{array}{lll}
\displaystyle \psi''(r) + \frac{N-1}{r} \psi^{\prime}(r) \ = \ -
\frac{m-1}{m} \psi^{\frac{1}{m-1}}(r) \qquad {\it for \ all} \ 0<
r < R_*,
\vspace{2mm} \\
\psi(0) = \alpha, \quad \psi^{\prime}(0) = 0
\end{array}
\right.\label{le}
\end{eqnarray}
such that
 \begin{eqnarray} \label{monotone-dec-phi}
&&    \psi(r)>0 \ \ {\it and} \ \  \psi^{\prime} (r) \ < \ 0 \quad {\it for \ all} \ r \in (0,R_*);  \\
\label{asymp-phi} && r^{-1} \psi^{\prime}(r) \ \longrightarrow \ -
\frac{(m-1)\alpha^{\frac{1}{m-1}}}{Nm} \qquad
{\it as} \ r \to 0; \\
&& \mbox{If } R_*<\infty  \mbox{ then } \lim_{r \to R_*} \psi(r) \
= \ 0. \label{cmp-sp}
\end{eqnarray}
\end{lemma}

\begin{proof}
Let us first realize that we can reduce the construction of
solutions of \eqref{le} to the construction of solutions of an integral
equation. Let us assume that the solution of \eqref{le} has the form
\begin{equation} \label{def-psi-integral}
\psi(r) := \alpha + \int_0^r W(s) \ ds\,,
\end{equation}
with $W\leq 0$ and $W\in L^1(0,\delta)$ for some $\delta >0$ to be
fixed later. Let us check that if $W$ is the solution of the
integral equation
\begin{equation}\label{def-WOrd}
W(r) = -\frac{m-1}{mr^{N-1}} \int_0^r s^{N-1}
                 \Big( \alpha + \int_0^s W(\sigma) \ d \sigma  \Big)^{\frac{1}{m-1}} \ ds.
\end{equation}
Then $\psi$ defined by (\ref{def-psi-integral}) is a solution of
\eqref{le}. Indeed, if $W\in L^1(0,\delta)$ and $W\leq 0$ then
\eqref{def-WOrd} implies that
\begin{equation}\label{tech1}
\|W\|_{L^\infty(0,\delta)}\leq \frac{(m-1)\delta}{Nm} \Big( \alpha
+ \|W\|_{L^1(0,\delta)} \Big)^{\frac{1}{m-1}} \,,
\end{equation}
and that $\psi(r)\geq \alpha-\|W\|_{L^\infty(0,\delta)}>0$ for
$\delta$ small enough. We also deduce that $W\in C^1((0,\delta))$
and thus $\psi\in C^2((0,\delta))$ by the fundamental theorem of
calculus. Taking derivatives on (\ref{def-psi-integral}), we get
\begin{align*}%\label{phi-double}
\psi^{\prime \prime}(r) + \frac{N-1}{r} \psi^{\prime}(r) = &\,
W^{\prime}(r) + \frac{N-1}{r} W(r)
\nonumber \\
= &\,-\frac{m-1}{m} \psi^{\frac{1}{m-1}}(r) +
\frac{(m-1)(N-1)}{mr^N} \int_0^r s^{N-1} \psi^{\frac{1}{m-1}}(s) \
ds
\nonumber \\
&\,- \frac{(N-1)(m-1)}{mr^{N}} \int_0^r s^{N-1}
                 \psi^{\frac{1}{m-1}}(s) \ ds
\nonumber \\
= &\, -\frac{m-1}{m} \psi^{\frac{1}{m-1}}(r), \qquad r>0.
\end{align*}
In fact, it is also easy to verify that $W^\prime\in
L^\infty(0,\delta)$. Therefore, the function $\psi$ defined by
(\ref{def-psi-integral}) belongs to $W^{2,\infty}(0,\delta)$ and
thus it belongs to $C^1([0,\delta])$ by Sobolev embeddings. Let us
now check that the initial conditions are met: $\psi(0)=\alpha$
and $\psi^{\prime}(0)=0$. Indeed, the integrability of $W$ and
\eqref{tech1} implies that
\begin{align*}
\lim_{r \to 0} \left| \psi(r) - \alpha \right| & = \lim_{r \to 0}
\left| \alpha + \int_0^r W(s) \ ds - \alpha \right| \ \le \
\lim_{r \to 0} \int_0^r |W(s)| \ ds \ = \ 0,
\nonumber \\
\lim_{r \to 0} \left| \psi^{\prime}(r) \right| & =  \lim_{r \to 0}
\left| W(r) \right| \ \leq \ \lim_{r \to 0} \frac{(m-1)r}{Nm}
\Big( \alpha + \|W\|_{L^1(0,r)} \Big)^{\frac{1}{m-1}} =  0.
\end{align*}
Thus we find that a function $\psi$ defined by
(\ref{def-psi-integral}) with $W\leq 0$ an integrable solution to
(\ref{def-WOrd}) gives a solution of \eqref{le}. It is trivial to verify
that the converse is also true, meaning that given a solution to
\eqref{le} and defining $W(r)=\psi^\prime(r)$ then $W(r)$ is a
nonpositive integrable solution of (\ref{def-WOrd}).

Let us now show the existence of solution to (\ref{def-WOrd}). For
$\varepsilon, \delta>0$, we introduce $X_{\varepsilon,\delta}$ by
\begin{eqnarray*}
X_{\varepsilon,\delta} & := &
    \left\{ g \in L^1(0,\delta); \ \|g\|_{L^1(0,\delta)} \le \varepsilon, \
            g(r) \le 0 \ {\rm for \ all} \ r \in [0,\delta)
            \right\}\,,
\end{eqnarray*}
where $\delta$ will be chosen later on small enough depending on
$\varepsilon$ and $\alpha$. We define the operator $F$ by
\begin{eqnarray*}
F: W(r) \in X_{\varepsilon,\delta} \ \longmapsto \
          -\frac{m-1}{mr^{N-1}} \int_0^r s^{N-1}
                 \Big( \alpha + \int_0^s W(\sigma) \ d \sigma  \Big)^{\frac{1}{m-1}} \ ds.
\end{eqnarray*}
This operator is well-defined since by \eqref{tech1}, we get
\begin{equation}\label{tech2}
\alpha + \int_0^s W(\sigma) \ d \sigma \geq
\alpha-\|W\|_{L^\infty(0,\delta)}\geq
\alpha-\frac{(m-1)\delta}{Nm} \Big( \alpha + \varepsilon
\Big)^{\frac{1}{m-1}} >0\,,
\end{equation}
by choosing
$$
\delta <
\min\left\{1, \ \frac{Nm\alpha}{(m-1)(\alpha+\varepsilon)^{\frac{1}{m-1}}} \right\} \,.
$$
Let us prove that $F(X_{\varepsilon,\delta}) \subset
X_{\varepsilon,\delta}$ for suitable $\delta$. Indeed, since $W
\in X_{\varepsilon,\delta}$, it holds from \eqref{tech2} that
$F(W)\leq 0$. Moreover, we get
\begin{eqnarray*}
\int_0^{\delta} |F(W(s))| \ ds & = & \int_0^{\delta}
             \frac{m-1}{ms^{N-1}}
                    \left|
             \int_0^s \sigma^{N-1}
                 \Big( \alpha + \int_0^{\sigma} W(\tau) \ d \tau  \Big)^{\frac{1}{m-1}} \ d\sigma
                    \right| \ ds
\nonumber \\
& \le & \int_0^{\delta}
             \frac{m-1}{ms^{N-1}}
             \int_0^s \sigma^{N-1}
                   (\alpha+\varepsilon)^{\frac{1}{m-1}}
                     \ d\sigma
                     \ ds = \frac{(m-1)}{2mN}
             \delta^2
                   (\alpha+\varepsilon)^{\frac{1}{m-1}} .
\end{eqnarray*}
Thus, the operator $F$ maps $X_{\varepsilon,\delta}$ onto itself
if we choose
\begin{equation}\label{def-delta1}
\delta
  \leq \delta_1:=\min\left\{1, \ \frac{Nm\alpha}{(m-1)(\alpha+\varepsilon)^{\frac{1}{m-1}}}, \ \left(\frac{2 \varepsilon Nm}{(m-1)(\alpha+\varepsilon)^{\frac{1}{m-1}}}
  \right)^{\frac{1}{2}}\right\}\,.
\end{equation}
Next, we shall show that $F$ is a contraction map from
$X_{\varepsilon,\delta}$ into itself for suitable small $\delta$.
We split this in two cases. Let us first show it for $2-2/N<m\leq
2$. In this range the function $x^p$ with $p=1/(m-1)$ is Lipschitz
and satisfies
$$
|a^p-b^p|\leq 2^{p-1}p \max\{a^{p-1},b^{p-1}\}|a-b|
$$
for all $a,b\geq 0$. Thus, we deduce that
\begin{align}
\left|
               \Big( \alpha + \int_0^{\sigma} W_1(\tau) \ d \tau  \Big)^{\frac{1}{m-1}}
              \right.-&\,\left.
               \Big( \alpha + \int_0^{\sigma} W_2(\tau) \ d \tau  \Big)^{\frac{1}{m-1}}
                   \right|\nonumber\\
 &\le \frac{2^{\frac{2-m}{m-1}}}{m-1} ( \alpha + \varepsilon )^{\frac{2-m}{m-1}}
\int_0^{\sigma} |W_1(\tau) - W_2(\tau)| \ d \tau\,.
\label{onecase}
\end{align}
In the second case, $m>2$, we use the mean value theorem together
with \eqref{tech1} to get
\begin{align}
\left|
               \Big( \alpha + \int_0^{\sigma} \right.&\left.W_1(\tau) \ d \tau  \Big)^{\frac{1}{m-1}}
              -\,
               \Big( \alpha + \int_0^{\sigma} W_2(\tau) \ d \tau  \Big)^{\frac{1}{m-1}}
                   \right|
\nonumber \\
& \le
\left(\alpha-\max\{\|W_1\|_{L^\infty(0,\delta)}, \|W_2\|_{L^\infty(0,\delta)}\} \right)^{\frac{2-m}{m-1}}
\int_0^{\sigma} |W_1(\tau) - W_2(\tau)| \ d \tau
\nonumber \\
& \le  \left(\alpha- \frac{(m-1)\delta}{
Nm}(\alpha+\varepsilon)^{\frac{1}{m-1}} \right)^{\frac{2-m}{m-1}}
\int_0^{\sigma} |W_1(\tau) - W_2(\tau)| \ d \tau\,.
\label{secondcase}
\end{align}
Putting together \eqref{onecase} and \eqref{secondcase} and taking
into account \eqref{def-delta1}, we get
$$
\left|
               \Big( \alpha + \int_0^{\sigma} W_1(\tau) \ d \tau  \Big)^{\frac{1}{m-1}}
              -
               \Big( \alpha + \int_0^{\sigma} W_2(\tau) \ d \tau  \Big)^{\frac{1}{m-1}}
                   \right|
\leq L_{\alpha,\varepsilon} \int_0^{\sigma} |W_1(\tau) -
W_2(\tau)| \ d \tau,
$$
where $L_{\alpha,\varepsilon}$ is given by
$$
 L_{\alpha,\varepsilon} :=
\max \left\{ \frac{2^{\frac{2-m}{m-1}}}{m-1} ( \alpha + \varepsilon
)^{\frac{2-m}{m-1}} , \left(\alpha- \frac{(m-1)\delta_1}{
Nm}(\alpha+\varepsilon)^{\frac{1}{m-1}} \right)^{\frac{2-m}{m-1}}
\right\}.
$$
Hence, we conclude that
\begin{align*}
 \|F&(W_1) - F(W_2) \|_{L^1(0,\delta)}
 = \int_0^{\delta} |F(W_1(s)) - F(W_2(s))| \ ds
\nonumber \\
& \le \int_0^{\delta}
             \frac{m-1}{ms^{N-1}}
             \int_0^s \sigma^{N-1}
               \left|
               \Big( \alpha + \int_0^{\sigma} W_1(\tau) \ d \tau  \Big)^{\frac{1}{m-1}}
              -
               \Big( \alpha + \int_0^{\sigma} W_2(\tau) \ d \tau  \Big)^{\frac{1}{m-1}}
                   \right| \ d\sigma
                     \ ds
                     \nonumber \\
& \le \int_0^{\delta}
             \frac{m-1}{ms^{N-1}}
             \int_0^s \sigma^{N-1} L_{\alpha,\varepsilon}
      \displaystyle \int_0^{\sigma} |W_1(\tau) - W_2(\tau)| \ d \tau
                \ d\sigma
                     \ ds
\nonumber \\
& \le
     \frac{(m-1)L_{\alpha,\varepsilon}}{2Nm} \delta^2
                       \|W_1 - W_2\|_{L^1(0,\delta)} \,.
\end{align*}
By choosing now
\begin{equation*}%\label{def-delta2}
\delta
  =
  \min \left\{\delta_1,
\left(\frac{Nm}{L_{\alpha,\varepsilon}(m-1)}\right)^{\frac{1}{2}}\right\}
\,,
\end{equation*}
we finally obtain
$$
 \|F(W_1) - F(W_2) \|_{L^1(0,\delta)}\leq \frac12 \|W_1 -
 W_2\|_{L^1(0,\delta)}\,,
$$
as desired. Therefore, the contraction mapping theorem yields the
existence and uniqueness of solution $W_*$ in
$X_{\varepsilon,\delta}$ of (\ref{def-WOrd}). We shall show
(\ref{monotone-dec-phi}). Since \eqref{le} can be rewritten as
\begin{eqnarray}\label{phi-radial}
 \frac{1}{r^{N-1}} \Big(r^{N-1} \psi^{\prime}(r) \Big)^{\prime}
       =
-\frac{m-1}{m} (\psi(r))^{\frac{1}{m-1}} \ < \ 0,
\end{eqnarray}
we infer that $\psi^{\prime}(r)<0$ for all $0<r<\delta$. Thus we
obtain (\ref{monotone-dec-phi}) on $[0,\delta)$. A standard
extension argument for ordinary differential equations proves the
existence of a solution satisfying the stated properties in
\eqref{monotone-dec-phi} for all $r$ as long as $\psi(r)>0$.
Therefore, we get a solution up to a maximal, possibly infinity at
this stage, $0<R_*\leq \infty$ satisfying \eqref{cmp-sp}. Finally,
we show (\ref{asymp-phi}). We can write
\begin{align*}
r^{-1} \psi^{\prime}(r) + \frac{(m-1) \alpha^{\frac{1}{m-1}}}{Nm}
& =  - \frac{m-1}{mr^{N}} \int_0^r s^{N-1} \psi^{\frac{1}{m-1}}(s)
\ ds + \frac{(m-1) \alpha^{\frac{1}{m-1}}}{Nm}
\nonumber \\
& = - \frac{m-1}{m} \left( \frac{1}{r^N} \int_0^r s^{N-1}
                 (\psi^{\frac{1}{m-1}}(s)-\psi^{\frac{1}{m-1}}(0)) \ ds
\right)\,,
\end{align*}
and thus, we deduce
$$
\left|r^{-1} \psi^{\prime}(r) + \frac{(m-1)
\alpha^{\frac{1}{m-1}}}{Nm} \right| \le \frac{m-1}{Nm}
\sup_{0<s<r}
   \left| \psi^{\frac{1}{m-1}}(0)- \psi^{\frac{1}{m-1}}(s) \right|
                  \ ds
\ \to \ 0 \qquad {\rm as} \ r \to 0\,,
$$
as claimed.
\end{proof}

We are now going to show that the maximal existence interval for
the solution constructed in the previous lemma is finite.

\begin{lemma}\label{lemma:local-existence2}
Let $N \ge 3$ and $m > 2-\frac{2}{N}$. For every $\alpha>0$, it
holds that $0<R_*(\alpha)<\infty$ for the solution obtained in
Lemma {\rm\ref{lemma:local-existence}}.
\end{lemma}

\begin{proof} To prove that $R_*$ is finite, we need to study the phase plane of
the dynamical system associated to \eqref{le} in detail. Let us consider
the following transformation $(u,v)$ introduced in \cite{BU}
\begin{eqnarray}\label{u-v}
u(r) & = & - \frac{m-1}{m} \frac{r \psi^{\frac{1}{m-1}}(r)}
{\psi^{\prime}(r)} \quad {\rm and} \quad v(r) \ = \ - \frac{r
\psi^{\prime}(r)}{\psi(r)}.
\end{eqnarray}
After some straightforward computations, if $\psi(r)$ is the
solution obtained in Lemma \ref{lemma:local-existence} then
$(u,v)$ satisfies the dynamical system
\begin{equation}\label{sysuv}
  r\frac{d u}{dr} = u\left(N-u-\frac{v}{m-1}\right), \qquad
  r\frac{d v}{dr} = v(-(N-2) + u + v)\,
\end{equation}
with initial conditions $u(0)=N$ and $v(0)=0$ due to
\eqref{asymp-phi}. Note that the dependence on $\alpha$ disappears
due to the transformation \eqref{u-v}. The claim will follow by
showing that there exists $0<R_*<\infty$ satisfying
$$
\lim_{r \to R_*} v(r) = +\infty \qquad \mbox{ and that } \qquad
\sup_{0<r<R_*} |\psi^{\prime}(r)| < \infty\,.
$$
We divide the proof into four steps. We included Figure
\ref{figdynsys} to show a sketch of the different steps in the
proof of this Lemma.

%**Picture**********************
\begin{figure}[ht]
\centering
\includegraphics[width=11cm]{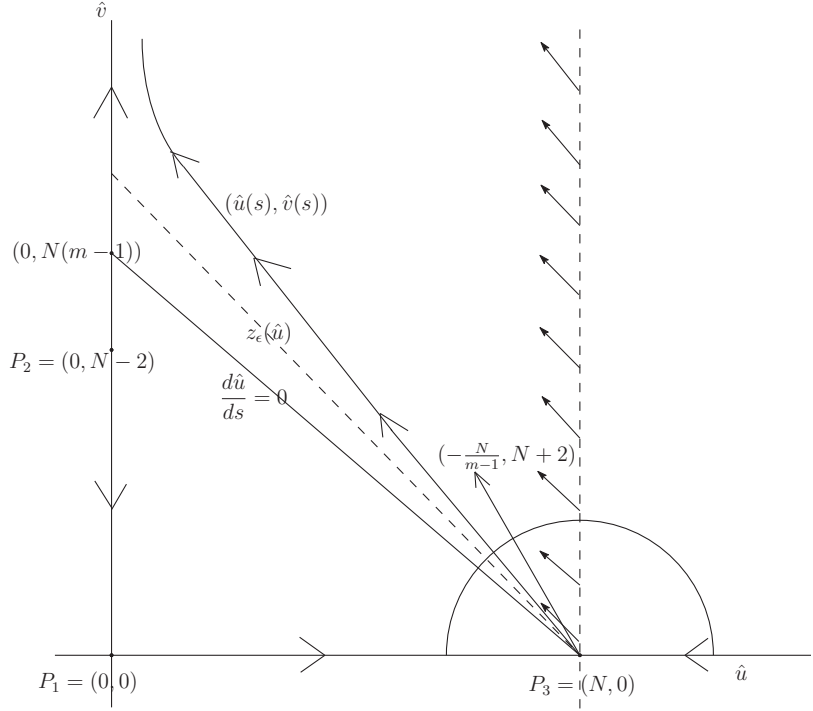}
\caption{Phase diagram associated to the dynamical system
\eqref{u-v-hat} with a sketch of the strategy of proof.}
\label{figdynsys}
\end{figure}
%***********************

{\it Step 1: Autonomous System.} The system \eqref{sysuv} can
easily be made autonomous by introducing new time scales defining
$\hat{u}$ and $\hat{v}$ as
\begin{eqnarray}\label{u-v-hat}
\hat{u} (s)  =  u(e^s), \quad \hat{v} (s)  =   v(e^s), \quad
   \mbox{with } s  = \log r.
\end{eqnarray}
Then we have by a direct calculation that the solution obtained in
Lemma \ref{lemma:local-existence} leads to a solution
$(\hat{u}(s),\hat{v}(s))$ defined on $(-\infty,S_*)$ of
\begin{eqnarray}
  \frac{d \hat{u}}{ds}  &=& \hat{u}\left(N-\hat{u}-\frac{\hat{v}}{m-1}\right), \qquad
  \frac{d \hat{v}}{ds}  \ = \ \hat{v}(-(N-2) + \hat{u} +
  \hat{v})\,,
\label{v-ds}
\end{eqnarray}
such that
$$
\lim_{s\to-\infty}\hat{u}(s)=N \qquad \mbox{and} \qquad
\lim_{s\to-\infty}\hat{v}(s)=0
$$
with $R_*=e^{S_*}$. It is trivial to check that both axis
$\hat{u}=0$ and $\hat{v}=0$ are all parts of trajectories of
solutions to \eqref{v-ds}. Therefore, the first quadrant, i.e.,
the set $\hat{u}\geq 0$ and $\hat{v}\geq 0$, is invariant for the
dynamical system. This autonomous dynamical system has three
stationary points in the first quadrant given by
\begin{equation*}%\label{w-stable-pt}
P_1  =  (0,0), \quad P_2 =   (0,N-2), \quad P_3 =  (N,0) .
\end{equation*}
since the fourth root of the system does not belong to the first
quadrant for $m>2-\tfrac{2}{N}$. The jacobian of the autonomous
system is
$$
J(\hat{u},\hat{v})=\left(\begin{array}{cc}
N-2\hat{u}-\frac1{m-1}\hat{v} & -\frac1{m-1}\hat{u} \\
\hat{v} & -(N-2)+\hat{u}+2\hat{v}
\end{array}\right)\,,
$$
then we deduce that $P_1$ and $P_3$ are unstable saddle points and
$P_2$ is a unstable nodal point. In fact, the eigenvalues are:
$\{N,-(N-2)\}$ for $P_1$, $\{(N-\frac{N-2}{m-1},N-2\}$ for $P_2$,
and $\{-N,2\}$ for $P_3$. Since the solution of interest verifies
that
\begin{eqnarray*}%\label{P3}
(\hat{u}(s), \hat{v}(s)) & \longrightarrow & P_3 \qquad {\rm as} \
s \to -\infty\,,
\end{eqnarray*}
then we need to check the local dynamics around $P_3$. Computing
eigenvectors associated to the eigenvalues of $J(P_3)$, we obtain
$(1,0)$ for the eigenvalue $-N$ and $(-\frac{N}{m-1},N+2)$ for
the eigenvalue $2$. The Hartman-Grosman theorem implies that
locally near $P_3$ the dynamics of \eqref{v-ds} are similar to the
linearized dynamics. Since the second eigenvector points
north-west at $P_3$ and it corresponds to the unstable manifold,
we deduce that there exists $\tilde s>0$ large enough such that
\begin{equation}\label{localdec}
\frac{d \hat{u}}{ds}(s)<0 \quad \mbox{ for } s \leq -\tilde s\,,
\end{equation}
for the solution corresponding to $\psi(r)$.
\vspace{2mm} \\
{\it Step 2: $\hat{u}$ is strictly decreasing and $\hat{v}$
diverges as $s\to S_*$.} With this aim, we show first
\begin{equation}\label{n-le-N}
  \hat{u}(s) \leq N \qquad {\rm for \ all \ } s \in (-\infty,S_*).
\end{equation}
Let us prove \eqref{n-le-N} by contradiction. If there exists
$s_0<S_*$ such that $\hat{u}(s_0) > N$ then take
$$
\tilde s_0 =\inf\{s<S_* \mbox{ such that } \hat{u}(s) > N\}.
$$
By construction and \eqref{localdec},
it verifies $-\infty<\tilde s_0\leq s_0<S_*$,
$\hat{u}(\tilde s_0) = N$, and $\frac{d \hat{u}}{ds}(\tilde{s_0})\geq 0$
since otherwise it will contradict its
definition. Since $-\infty<\tilde s_0$ by \eqref{localdec}, we
know that $\hat{v}(\tilde s_0)>0$. This contradicts the equation
for $\hat{u}$ in \eqref{v-ds} since it is easy to check that in
the vertical line $(N,\hat{v})$ with $\hat{v}>0$, the value of
$\frac{d \hat{u}}{ds}$ is strictly negative. This shows the claim
\eqref{n-le-N}.

The next idea is to show that the solution associated to $\psi(r)$
never touches the straight line $N-\hat{u}-\frac{\hat{v}}{m-1}=0$
and thus $\frac{d \hat{u}}{ds}$ is strictly negative for all $s\in
(-\infty,S_*)$. Let us introduce the comparison function
$z_{\varepsilon}$ defined by
\begin{eqnarray*}% \label{def-ze}
z_{\varepsilon} (\hat{u}) = (m-1)(1+ \varepsilon)(N-\hat{u})\,,
\end{eqnarray*}
with $\varepsilon >0$ to be chosen later. Since $N \ge 3$ and $m
\geq 2- \frac{2}{N}$, we have
\begin{equation*}% \label{z-0}
z_{\varepsilon} (0) = (m-1)N(1 + \varepsilon) \ge (N-2)(1 +
\varepsilon) > N-2.
\end{equation*}
By defining the function $f$ as
$$
 f(\hat{u}, \hat{v})
    = \frac{\hat{v} (-(N-2) + \hat{u} + \hat{v}) }
    {\hat{u} (N-\hat{u} - \frac{1}{m-1} \hat{v}) }\,,
$$
then the curve $(\hat{u}(s),\hat{v}(s))$ can parameterized in
terms of $\hat{u}$ due to \eqref{localdec} in a time interval
$(-\infty,\tilde s]$ for small enough $\tilde s$ and it satisfies
\begin{eqnarray} \label{vu-f}
  \frac{d\hat{v}}{d \hat{u}}
    = f(\hat{u}, \hat{v}) \qquad \mbox{ with } \hat{v}(N)=0\,.
\end{eqnarray}
After some computations, one can verify that
\begin{eqnarray} \label{fz-zprime}
  f(\hat{u}, z_{\varepsilon})
  \le  \frac{dz_{\varepsilon} (\hat{u})}{d \hat{u}}\,,
\end{eqnarray}
where $\varepsilon$ is chosen as
\begin{eqnarray*}
\varepsilon = \left\{
\begin{array}{lll}
 \displaystyle \frac{2-m}{m} & \quad {\rm for} \ 2-\frac{2}{N} < m < 2,
 \\[3mm]
 \displaystyle \frac2N & \quad {\rm for} \ m \ge 2.
\end{array}
\right.
\end{eqnarray*}
Indeed, (\ref{fz-zprime}) is equivalent to
\begin{eqnarray} \label{equiv-fz}
(2-m)\hat{u} +  (m-1)N \ge N-2 + \varepsilon  \Big( m\hat{u} -
(m-1)N \Big).
\end{eqnarray}
In case $2-\frac{2}{N} < m < 2$, a direct calculation shows
(\ref{equiv-fz}). In case $m \ge 2$, (\ref{equiv-fz}) results from
$$
(2-m)\hat{u} +  (m-1)N \ge (2-m)N +  (m-1)N =N \geq N-2 +
\varepsilon \Big( m\hat{u} - (m-1)N \Big)
$$
since $\hat{u}\leq N$ from \eqref{n-le-N}.
Based on (\ref{vu-f}),
(\ref{fz-zprime}), and $z_{\varepsilon}(N)=0$, we find by the
comparison principle for first-order ODEs that $\hat{v}(\hat{u})
\ge z_{\varepsilon}(\hat{u})$ in the interval $(\hat{u}(\tilde
s),N)$. It is obvious that the argument can be now continued as
long as the solution $\hat{u}$ exists, and thus $\hat{v}(\hat{u})
\ge z_{\varepsilon}(\hat{u})$ in the interval $(\hat{u}(\tilde
s),N)$ for all $s\in (-\infty,S_*)$. This implies that the
solution never touches the straight line
$N-\hat{u}-\frac{\hat{v}}{m-1}=0$ as claimed, and thus
\begin{equation*}
\frac{d \hat{u}}{ds}(s)< 0 \quad \mbox{and} \quad
  \hat{u}(s) < N \qquad {\rm for \ all \ } s \in (-\infty,S_*).
\end{equation*}
On the other hand, since there is no stable point inside the
first quadrant, we finally find
\begin{eqnarray*}%\label{n-le-N2}
\lim_{s\to S_*}\hat{v}(s) = +\infty\,,
\end{eqnarray*}
and in particular there exists $0<\bar s<\infty$ such that
\begin{eqnarray}\label{vs-star}
\hat{v}(\bar s) > N-2.
\end{eqnarray}
{\it Step 3: $S_*<\infty$.} Indeed, (\ref{vs-star}) implies that
\begin{align*}
\Big( \frac{d\hat{v}}{ds} \Big)(\bar s) & = \hat{v}(\bar s)
(-(N-2) + \hat{u}(\bar s) + \hat{v}(\bar s)) \ge \hat{v}(\bar s)
(\hat{v}(\bar s)-(N-2))
\nonumber \\
& \ge (\hat{v}(\bar s)-(N-2))^2 > 0.
\end{align*}
Hence, there exists $s^{\prime}>\bar s$ such that
\begin{eqnarray*}
\hat{v}(s) & > & N-2, \quad \Big( \frac{d\hat{v}}{ds} \Big)(s) \
\ge \ (\hat{v}(s)-(N-2))^2 \ > \ 0 \qquad {\rm for \ all \ } \bar
s < s < s^{\prime}.
\end{eqnarray*}
Repeating this procedure, we find that
\begin{eqnarray*}
\hat{v}(s) & > & N-2,
\quad
\Big( \frac{d\hat{v}}{ds} \Big)(s)
\ \ge \ (\hat{v}(s)-(N-2))^2 \ > \ 0
 \qquad {\rm for \ all \ } \bar s < s < S_*\,.
\end{eqnarray*}
On the other hand, we consider the following ordinary differential equation:
\begin{eqnarray*}
\frac{d \hat{w}} {ds} & = & (\hat{w}(s) - (N-2))^2, \qquad
\hat{w}(\bar s) = \hat{v}(\bar s).
\end{eqnarray*}
Since it holds that
\begin{eqnarray*}
\hat{w}(s)
 & = & \frac{\hat{w}(\bar s)-(N-2)}{1+(\bar s-s)(\hat{w}(\bar s)-(N-2))} + N-2,
\end{eqnarray*}
we find that
\begin{eqnarray*}
\hat{w}(s) & \longrightarrow & \infty \qquad {\rm as} \ \  s \to
\frac{1+\bar s(\hat{w}(\bar s)-(N-2))}{\hat{w}(\bar s)-(N-2)}
=\bar S \,.
\end{eqnarray*}
Therefore, the comparison principle yields that $S_*\leq \bar
S<\infty$.
\vspace{2mm} \\
{\it Step 4: $\psi(R_*)=0$.} The previous step already shows that
$R_*<\infty$ but to finish the proof we need to get the control of
$\psi^\prime(r)$ to derive finally that $\psi(R_*)=0$. To this
end, we go back to (\ref{phi-radial}) to get
\begin{eqnarray*}
[r^{N-1} \psi^{\prime}(r)]_{0}^{r}
\ = \
\int_0^{r} \Big(s^{N-1} \psi^{\prime}(s) \Big)^{\prime} \ ds
 & = &  -\frac{m-1}{m}
            \int_0^{r} s^{N-1} \psi^{\frac{1}{m-1}}(s) \ ds,
\end{eqnarray*}
on $0<r<R_*$. Thus, we have
\begin{equation*} %\label{phi-est}
r^{N-1} \left| \psi^{\prime}(r) \right|
 = \frac{m-1}{m}
           \left|
                \int_0^{r} s^{N-1} \psi^{\frac{1}{m-1}}(s) \ ds
           \right| \leq \frac{(m-1)r^N\alpha^{\frac{1}{m-1}}}{mN}
           \,
\end{equation*}
since $\psi(r) < \psi(0) = \alpha$ due to $\psi^{\prime}(r) <0$ in
Lemma \ref{lemma:local-existence}. Thus we conclude
\begin{equation*}%\label{phi-fininte}
\sup_{0<r<R_*} \left| \psi^{\prime}(r) \right| \le
\frac{(m-1)R_*\alpha^{\frac{1}{m-1}}}{mN} \,,
\end{equation*}
completing the proof of Lemma \ref{lemma:local-existence2}.
\end{proof}

Let us notice, as in \cite[Lemma 15]{S-Z}, that the previous lemma
immediately implies that the mapping between the initial value
$\alpha$ and the mass $M$ is bijective if $m>2-\frac{2}{N}$.

\begin{cor}
\label{cor:M-alpha} Let $N \ge 3$ and $m>2-\frac{2}{N}$. For every
$M>0$, there exist a unique $\alpha=\alpha(M)>0$ such that the
solution $\psi$ of \eqref{le} given by Lemma
{\rm\ref{lemma:local-existence}} fulfills that
\begin{eqnarray*}%\label{cor:mass}
\omega_N \int_0^{R_*(\alpha)} \psi^{\frac{1}{m-1}}(s) s^{N-1} \ ds
& = & M.
\end{eqnarray*}
\end{cor}

\begin{proof}
Given the solution $\Psi$ of \eqref{le} with the initial data
$\Psi(0)=1$, one can easily check that $\psi$ defined as
$$
\psi(r) := \alpha \Psi(\alpha^\mu r) \quad \mbox{with} \
\mu=\frac{1}{2}\left(\frac{2-m}{m-1}\right)\,,
$$
for every $\alpha >0$ is the solution of \eqref{le} with
$\psi(0)=\alpha$. This readily implies $R_*(\alpha)=\alpha^{-\mu}
R_*(1)$. For $\alpha>0$, we define $M_1(\alpha)$ by
$$
M_1(\alpha) := \omega_N \int_0^{R_*(\alpha)} \left[\alpha
\Psi(\alpha^\mu r)\right]^{\frac{1}{m-1}} r^{N-1} \ dr= \omega_N
\alpha^{\frac{1}{m-1}-\mu N} \int_0^{R_*(1)}
\Psi(s)^{\frac{1}{m-1}} s^{N-1} \ ds.
$$
Since $\frac{1}{m-1}-\mu N>0$ under our assumptions, then we
deduce the result.
\end{proof}

%\begin{remark}
{\bf Remark.}
Lemmas
{\rm\ref{lemma:local-existence}-\ref{lemma:local-existence2}} are
perfectly valid for the critical case $m=2-\tfrac{2}{N}$. However,
previous corollary implies that all stationary solutions have
equal mass since $\frac{1}{m-1}-\mu N=0$ in this case. Therefore,
there are infinitely many stationary radial solutions only for a
single value of the total mass $M$, the critical mass, as already
proved in {\rm\cite{BCL}}.
%\end{remark}

%%%%%%%%%%%%%%%%%%%%%%%%%%%%%%%%%%%%%%%%%%%%%%%%%%%%%%%%%%%%%%%%%%%%%%%%%%%%%%%%%%%%%%%%%%%%%%%%%%%%%%%%5

\section{Global minimizers: Existence, Uniqueness, and Qualitative Properties.}
\label{section:existence-stationary-point}
\setcounter{equation}{0}

In this section we will show that global minimizers of the free energy \eqref{def-EU} exist in an adequate functional space, and we will characterize them by a nonlocal nonlinear integral equation. This equation in turn will give us radial symmetry, compact support and uniqueness up to translations of the global minimizer of the free energy, since it has to coincide up to translations with the unique solution of the Lane-Emden equation with a given mass obtained in previous Section. This complements the variational information obtained by different methods in \cite{Kim-Yao} using \cite{LY}. We do a direct proof of this fact without resorting to techniques in \cite{LY}. Let us start by defining the functional space
\begin{equation*}%\label{def-WM}
\mathcal{Y}_M := \{U \in L^1 \cap L^m(\R^N);
      \ \|U\|_1 = M, \ U(x) \ge 0 \quad
{\rm for \ a.e.} \   x \in \R^N\}\,.
\end{equation*}
The main challenge is how to obtain control of the confinement of mass for minimizing sequences due to the translational invariance of the energy functional \eqref{def-EU}. We avoid to do this in the whole space $\R^N$ by following a strategy used in \cite{Strohmer,S-Z,CCV} in the three dimensional case. We first prove that the free energy is bounded below that together with compactness arguments will show the existence of global minimizers in the restricted functional setting of compactly supported functions given by
\begin{eqnarray*}%\label{def-WMR}
\mathcal{Y}_{M,R} & = & \{U \in \mathcal{Y}_M; \ U(x) =0 \quad
      {\rm for \ a.e.} \   |x| \ge R\}
\end{eqnarray*}
for all $R>0$. Calculus of variations arguments applied to the free energy gives us a necessary condition on global minimizers by a nonlocal integral equation whose support is restricted to the ball of radius $R$. At this point decreasing rearrangement techniques imply the radial symmetry and the uniqueness up to translations of the global minimizers in $\mathcal{Y}_{M,R}$. This together with the careful analysis of the Lane-Emden system \eqref{le} in previous section allow us to pass to the limit $R\to \infty$ in this minimization procedure leading to uniqueness up to translations of the global minimizer in $\mathcal{Y}_{M}$.

\subsection{The free energy is bounded below in $\mathcal{Y}_M$}

We start by reminding an inequality obtained by interpolation from
the classical Hardy-Littlewood-Sobolev (HLS) inequality.

\begin{lemma} \label{lem:HLS}
Let $N \ge 3$ and let $m \geq 2-\frac{2}{N}$. For every $f,g \in
L^1 \cap L^m(\R^N)$, it holds that
\begin{eqnarray*}
\left| \int_{\R^N} \int_{\R^N} \frac{f(x)g(y)}{|x-y|^{N-2}} \ dxdy \right| & \le
& C_{HLS} \|f\|_1^{1-\theta} \|f\|_m^{\theta} \|g\|_1^{1-\theta}
\|g\|_m^{\theta}
\end{eqnarray*}
with $\theta=\frac{(N-2)m}{2N(m-1)}$, where $C_{HLS}$ is the sharp
constant of the HLS inequality.
\end{lemma}

\begin{proof} H\"{o}lder inequality implies that
\begin{eqnarray}\label{hl-1}
\int_{\R^N} \int_{\R^N} \frac{f(x)g(y)}{|x-y|^{N-2}} \ dxdy & \le &
\|f\|_{\frac{2N}{N+2}} \||x|^{2-N}*g\|_{(\frac{2N}{N+2})^{\prime}}
\end{eqnarray}
with $p^{\prime}=\frac{p}{p-1}$. By the HLS inequality, there
exists a positive number $C_{HLS}=C_{HLS}(N)$ such that
\begin{eqnarray}\label{hl-2}
\||x|^{2-N}*g\|_{(\frac{2N}{N+2})^{\prime}} \ \le \ C_{HLS}
\|g\|_{\frac{2N}{N+2}}.
\end{eqnarray}
Since $1 < \frac{2N}{N+2} < m$, implied by $m \geq 2-\frac{2}{N} >
\frac{2N}{N+2}$, then interpolation in $L^p$-spaces gives
\begin{eqnarray}\label{hl-3}
\|g\|_{\frac{2N}{N+2}} & \le & \|g\|_1^{1-\theta}
\|g\|_m^{\theta}.
\end{eqnarray}
Putting together (\ref{hl-1}), (\ref{hl-2}), and (\ref{hl-3}), we
find the desired result.
\end{proof}

We now show that $E[U]$ is well-defined and bounded below for all
$U \in L^1 \cap L^m(\R^N)$ for $m > 2-\frac{2}{N}$ with $N \ge 3$.

\begin{prop} \label{prop:muM}
Let $N \ge 3$ and $m > 2-\frac{2}{N}$. For every $M>0$, it holds
that $E[U]<\infty$ for all $U \in \mathcal{Y}_M$ and
\begin{eqnarray*}
\mu_M & := &
\displaystyle
\inf_{U \in \mathcal{Y}_M} E[U]  \ > \ -\infty.
%\end{array}
%\right.
\end{eqnarray*}
\end{prop}

\begin{proof}
Given $m > 2-\frac{2}{N}$ and taking into account Lemma
\ref{lem:HLS} for exponent $2-\frac{2}{N}$, we get that
\begin{equation} \label{ent-inff4}
\left|\int_{\R^N} U(x)V(x) \ dx\right| \le \frac{C_{HLS} }{(N-2)
\omega_{N-1}}
   \|U\|_1^{2/N} \| U \|_{2-\tfrac{2}{N}}^{2-\tfrac{2}{N}} \leq \frac{C_{HLS} }{(N-2)
\omega_{N-1}}
   \|U\|_1^{2-2\theta} \|U\|_m^{2\theta}
\end{equation}
for all $U \in L^1 \cap L^m(\R^N)$, and thus $E[U]<\infty$ for all
$U\in \mathcal{Y}_M$. Using (\ref{ent-inff4}) into the free energy
definition (\ref{def-EU}), we deduce
\begin{align}
E[U] & \ge
 \frac{1}{m-1} \|U\|_m^m \ - \  \frac{C_{HLS} }{2(N-2) \omega_{N-1}}
     \|U\|_1^{2/N} \|U\|_{2-\tfrac{2}{N}}^{2-\tfrac{2}{N}} \label{E-1}\\
& \ge
 \frac{1}{m-1} \|U\|_m^m \ - \  \frac{C_{HLS} }{2(N-2) \omega_{N-1}}
     \|U\|_1^{2-2\theta} \|U\|_m^{2\theta} \label{E-2}
\end{align}
for all $U \in L^1 \cap L^m(\R^N)$ and $m > 2-\frac{2}{N}$. Let us
consider the functions $f_1$ and $f_2$ such as
\begin{eqnarray*}
f_1(r) & = & \frac{1}{m-1} r^m \quad {\rm and} \quad
f_2(r) \ = \ \frac{C_{HLS} }{2(N-2) \omega_{N-1}} M^{\frac{2}{N}} r^{2-\frac{2}{N}}.
\end{eqnarray*}
It should be noted that since $m > 2-\frac{2}{N}$, there exists
$R_0>0$ such that
\begin{eqnarray}\label{EU-lowerbd-3}
\left\{
\begin{array}{llll}
f_1(r) & > & f_2(r) \quad & {\rm for \ all} \ r > R_0, \\
f_1(r) & < & f_2(r) \quad & {\rm for \ all} \ 0 < r < R_0.
\end{array}
\right.
\end{eqnarray}
Using the decomposition $\R^N=A\cup(\R^N/A)$ with $A = \{x \in
\R^N; U(x) \leq R_0\}$, we have by (\ref{E-1}) that
\begin{equation*}
E[U] \ge \int_A \big( f_1(U(x)) -f_2(U(x)) \big) \ dx +
\int_{\R^N/A} \big( f_1(U(x)) -f_2(U(x)) \big) \ dx\,,
\end{equation*}
for all $U\in \mathcal{Y}_M$. Since the second term in the right-hand side
is positive and $f_1 \ge f_2$, it holds by (\ref{EU-lowerbd-3}) that
\begin{align*}%\label{EU-lowerbd-4}
E[U] & \ge \int_A \big( f_1(U(x)) -f_2(U(x)) \big) \ dx \ge -
\int_A \frac{C_{HLS} }{2(N-2) \omega_{N-1}}
                                M^{\frac{2}{N}} U^{2-\frac{2}{N}}(x) \ dx
\nonumber \\
& \ge - \frac{C_{HLS} }{(N-2) \omega_{N-1}}
      M^{\frac{2}{N}} R_0^{1-\frac{2}{N}}  \int_A U(x) \ dx
\ge - \frac{C_{HLS} }{(N-2) \omega_{N-1}} M^{\frac{2}{N}+1}
R_0^{1-\frac{2}{N}}\,,
\end{align*}
for all $U \in \mathcal{Y}_M$, proving the bound from below.
\end{proof}

Even if the free energy \eqref{def-EU} is bounded from below in $\mathcal{Y}_M$, we do not know how to show that minimizing sequences are compact in $\mathcal{Y}_M$ due to the lack of control of the escape of mass at infinity, compared to \cite{CCV} in the two dimensional case.
On the other hand, it is very easy to check due to Lemma \ref{lem:HLS} that $E[U]$ is a continuous functional with respect to the strong convergence in $L^1 \cap L^m(\R^N)$. In fact, this allows us to approximate the minimization problem in the whole class $\mathcal{Y}_M$ by restricting to compactly supported functions. To avoid the lack of control in the mass at infinity, we will first minimize the free energy among compactly supported densities in $\mathcal{Y}_{M,R}$. Let us define $\mu_{M,R}$ by
\begin{eqnarray*}% \label{def-muMR}
\mu_{M,R} & := &  \inf_{U\in \mathcal{Y}_{M,R}} E[U].
\end{eqnarray*}
Observe that Proposition \ref{prop:muM} implies that $\mu_{M,R} \ge \mu_M \ > \ -\infty$, it is a decreasing function of $R$ by construction, and because of the continuity of $E[U]$ in $L^1 \cap L^m(\R^N)$, we claim that
\begin{equation}\label{limmins}
\lim_{R\to\infty} \mu_{M,R} = \mu_M\,.
\end{equation}
To show that, it suffices to take a minimizing sequence $\{U_k\}_{k\in\N} \subset \mathcal{Y}_M$, so that $ \lim_{k\to \infty} E(U_k)=\mu_M$. We know that
$$
\lim_{R\to \infty} \left\|U_{k,R}-U_k\right\|_p=0 \quad \mbox{with } U_{k,R}:= U_k \frac{M}{\|U_k\|_{L^1(B_R)}}\chi_{B_R}\in \mathcal{Y}_{M,R}\,,
$$
for all $k\in\N$ and all $1\leq p\leq m$ by dominated convergence theorem. The continuity of $E$
implies that
$$
\lim_{R\to \infty} \mu_{M,R} \leq \lim_{R\to \infty} E[U_{k,R}] = E[U_k]
$$
for all $k\in\N$. So we conclude that
$$
\mu_{M} \leq \lim_{R\to \infty} \mu_{M,R}\leq \mu_{M}\,.
$$
%%%%%%%%%%%%%%%%%%%%%%%%%%%%%%%%%%%%%%%%%%%%%%%%%%%%%%%%%%%%%%%%%%%%%%%%%%

\subsection{Global minimizers of the free energy in balls}

We show first the existence of a radial minimizer of $E$ in $\mathcal{Y}_{M,R}$ and that all global minimizers are radial. Note that we do not know yet any uniqueness of radial global minimizer.

\begin{lemma} \label{lemma:ex-minimizer}
Let $N \ge 3$ and $m > 2-\frac{2}{N}$.
For every $M>0$ and $R>0$, there exists a radial function $U_R \in \mathcal{Y}_{M,R}$ such that
\begin{eqnarray} \label{ex-minimizer-1}
E[U_R] & = & \mu_{M,R}\,,
\end{eqnarray}
for which $V_R=\Gamma\ast U_R \in L_{loc}^{\frac{m}{m-1}}(\R^N)$. Moreover, all global minimizers are of the form $\tilde{U}_R(x+y)$ for all $y\in \R^N$ such that $\tilde{R}_o+|y|\leq R$ with $\tilde{U}_R$ being a radial global minimizer of $E$ in $\mathcal{Y}_{M,R}$ with support $\bar{B}_{\tilde{R}_o}$, $0<\tilde{R}_0\leq R$.
\end{lemma}

\begin{proof}
By the definition of $\mu_{M,R}$, there exists a sequence
$\{U_n\}_{n\in\N} \subset \mathcal{Y}_{M,R}$ such that
\begin{eqnarray} \label{EU-convergence}
E[U_n] & \longrightarrow & \mu_{M,R}
\qquad {\rm as} \ n \to \infty.
\end{eqnarray}
We have by (\ref{E-2}) that
\begin{eqnarray*}%\label{Un-sequence}
E[U_n]
& \ge &
 \frac{1}{m-1} \|U_n\|_m^m \ - \  \frac{C_{HLS} }{2(N-2) \omega_{N-1}}
     M^{2-\frac{(N-2)m}{N(m-1)}} \|U_n\|_m^{\frac{(N-2)m}{N(m-1)}}
\nonumber \\
& = &
\|U_n\|_m^{\frac{(N-2)m}{N(m-1)}}
\Big(
 \frac{1}{m-1} \|U_n\|_m^{m-\frac{(N-2)m}{N(m-1)}}
              \ - \  \frac{C_{HLS} }{2(N-2) \omega_{N-1}}
     M^{2-\frac{(N-2)m}{N(m-1)}}
\Big)
\end{eqnarray*}
for $U_n \in \mathcal{Y}_{M,R}$, where $C_{HLS}=C_{HLS}(N)$. Since
$\{E[U_n]\}_{n\in\N}$ is bounded, there exists $L>0$ such that
\begin{eqnarray}\label{bound-sequence}
\|U_n\|_m & \le & L \quad {\rm for \ all} \ n\in\N.
\end{eqnarray}
Let us denote by $f^\#$ the radially decreasing rearrangement of the function $f$ in $\R^N$.
Now, we make use of classical radially decreasing rearrangement inequalities in \cite[Lemma 2.1]{Lieb} to show that
\begin{equation}\label{dr}
\int_{\R^N}\int_{\R^N}U_n(x) \Gamma(x-y) U_n(y) \, dx\ dy \leq \int_{\R^N}\int_{\R^N}U_n^\#(x) \Gamma(x-y) U_n^\#(y) \, dx\ dy\, ,
\end{equation}
and thus $E[U_n]\geq E[U_n^\#]$ for all $n\in\N$. Therefore, we can assume without loss of generality that our minimizing sequence is composed of radial functions.

By virtue of (\ref{bound-sequence}), there exist a subsequence
$\{U_{n_j}\}_{j=1}^{\infty} \subset \mathcal{Y}_{M,R}$ such that
\begin{eqnarray}\label{weakconv}
U_{n_j}
&
\rightharpoonup
&
U_R
\hspace{6mm}
{\rm weakly}         \ {\rm in} \ L^m(\R^N).
\end{eqnarray}
In addition by weak convergence, it holds that its support lies in
$B_R$. We also observe that
\begin{eqnarray*}%\label{massUR}
M \ = \ \int_{\R^N} U_{n_j} \ dx = \int_{B_R} U_{n_j} \ dx
\longrightarrow \int_{B_R} U_R \ dx \ = \ \int_{\R^N} U_R \ dx \quad {\rm
as} \ j \to \infty,
\end{eqnarray*}
and thus, we prove that $U_R \in \mathcal{Y}_{M,R}$. Moreover, the lower
semi-continuity of norm yields that
\begin{eqnarray}\label{lowersemic}
\|U_R\|_m & \le & \liminf_{j \to \infty} \|U_{n_j}\|_m.
\end{eqnarray}
Let us denote $V_{n_j}=\Gamma \ast U_{n_j}$, i.e., $-\Delta
V_{n_j}=U_{n_j}$. Now, we shall prove the strong convergence of
$V_{n_j}$. We first treat the case of $2-\frac{2}{N} < m <
\frac{N}{2}$. Sobolev compactness theorem implies that there
exists a subsequence, still denoted by
$\{V_{n_j}\}_{j=1}^{\infty}$, such that
\begin{eqnarray*}%\label{strongconv,,}
V_{n_j}
&
\longrightarrow
&
\tilde{V}_R
\hspace{6mm}
{\rm strongly}         \ {\rm in} \ L^q(B_R) \quad {\rm for} \ q<\frac{Nm}{N-2m}.
\end{eqnarray*}
For $N \ge 3$ we have that $\frac{m}{m-1} < \frac{Nm}{N-2m}$, it
follows from the above convergence that
\begin{eqnarray}\label{strongconv}
V_{n_j}
&
\longrightarrow
&
\tilde{V}_R
\hspace{6mm}
{\rm strongly}         \ {\rm in} \ L^{\frac{m}{m-1}}(B_R).
\end{eqnarray}
In case $m \ge \frac{N}{2}$, it holds that
\begin{eqnarray*}%\label{strongconvll}
V_{n_j}
&
\longrightarrow
&
\tilde{V}_R
\hspace{6mm}
{\rm strongly}         \ {\rm in} \ L^q(B_R)
\end{eqnarray*}
for all $1<q<\infty$, and hence (\ref{strongconv}) holds for all
$m > 2-\frac{2}{N}$. Combining (\ref{weakconv}) with
(\ref{strongconv}), we observe that
\begin{eqnarray}
\lefteqn{
\left|
\int_{B_R} U_{n_j}(x)V_{n_j}(x) \ dx - \int_{B_R} U_R(x)\tilde{V}_R(x) \ dx
\right|
} \nonumber \\
& \le &
\left|
\int_{B_R} U_{n_j}(x) (V_{n_j}(x) - \tilde{V}_R(x) ) \ dx
\right|
+
\left|
\int_{B_R} (U_{n_j}(x) - U_R(x))\tilde{V}_R(x) \ dx
\right|
\nonumber \\
& \le &
\|U_{n_j}\|_m \| V_{n_j} - \tilde{V}_R\|_{L^{\frac{m}{m-1}}(B_R)}
+
\left|
\int_{B_R} (U_{n_j}(x) - U_R(x))\tilde{V}_R(x) \ dx
\right|
\nonumber \\
& \le &
L \| V_{n_j} - \tilde{V}_R\|_{L^{\frac{m}{m-1}}(B_R)}
+
\left|
\int_{B_R} (U_{n_j}(x) - U_R(x))\tilde{V}_R(x) \ dx
\right|
\nonumber \\
& \longrightarrow & 0 \qquad {\rm as} \ j \to \infty.
\label{uvj-uvr}
\end{eqnarray}
We next show that $\tilde{V}_R=V_R$ is given by
\begin{eqnarray} \label{Vtilde-G}
V_R (x) & = & \int_{\R^N} \Gamma(x-y) U_R(y) \ dy.
\end{eqnarray}
First, it should be noted that $\tilde{V}_R \in L^{\frac{2N}{N-2}}(\R^N)$ by the HLS inequality \eqref{hl-2} applied to the sequence $V_{n_j}$ using \eqref{bound-sequence}. We
deduce from (\ref{weakconv}) and (\ref{strongconv}) that
\begin{eqnarray*}
-\int_{\R^N} \tilde{V}_R(x) \Delta \varphi(x) \ dx & = & \int_{\R^N} U_R(x)
\varphi(x) \ dx
\end{eqnarray*}
for all $\varphi \in C_0^{\infty}(\R^N).$
Then by the Weyl Lemma, it holds
$\tilde{V}_R \in W_{loc}^{2,m}(\R^N)$
with
\begin{eqnarray*}
-\Delta \tilde{V}_R(x) & = & U_R(x) \quad {\rm a.e.} \ x \in \R^N.
\end{eqnarray*}
Since $\tilde{V}_R \in L^{\frac{2N}{N-2}}(\R^N)$, we conclude that $\tilde{V}_R=V_R$, i.e., $\tilde V_R$ is the Newtonian potential of $U_R$ given in (\ref{Vtilde-G}). Observe that due to $m>2-\frac2N$ then
$V_R \in L_{loc}^{\frac{m}{m-1}}(\R^N)$. Finally, since the Newtonian potentials $V_n$ are radial functions for all $n\in\N$ due to the radial symmetry of $U_n$, then $V_R$ is radially symmetric being the strong $L^p$ limit of radial functions, and thus $U_R$ too due to the regularity above.

Finally, let us make use of the free energy convergence
(\ref{EU-convergence}) and $U_R \in \mathcal{Y}_{M,R}$ together with the weak
lower semicontinuity (\ref{lowersemic}) and the strong convergence
(\ref{uvj-uvr}) to derive that
\begin{eqnarray*}
\mu_{M,R}
& = &
\lim_{j \to \infty} E[U_{n_j}]
\ = \
\liminf_{j \to \infty} E[U_{n_j}]
\nonumber \\
& \ge & \frac{1}{m-1} \|U_R\|_{L^m(B_R)}^m - \frac12\int_{B_R} U_R(x) V_R(x) \ dx
\nonumber \\
& = & E[U_R] \ge
\inf_{U \in \mathcal{Y}_{M,R}} E[U] \ = \ \mu_{M,R},
\end{eqnarray*}
which yields the existence of a radial global minimizer (\ref{ex-minimizer-1}).

To finish the proof, we need to show that all global minimizers are obtained in terms of radial global minimizers. This is a consequence of the equality cases in \eqref{dr} also discussed in \cite[Lemma 2.1]{Lieb}. The inequality in \eqref{dr} applied to a function $f$ is strict unless there exists $y\in\R^N$ such that $f(x)=f^\#(x+y)$. Therefore, assume that $f$ is another global minimizer of $E$ in $\mathcal{Y}_{M,R}$, then
$E[f]= E[f^\#]$ by definition of global minimizer. Since the $L^m$ norms of $f$ and $f^\#$ are equal, then $f$ satisfies the equality in \eqref{dr}. Finally, \cite[Lemma 2.1]{Lieb} ensures us that there exists $y\in\R^N$ such that $f(x)=f^\#(x+y)$. Denote $\mbox{supp}(f^\#)=\bar{B}_{\tilde{R}_o}$, then $\tilde{R}_o+|y|\leq R$ since $\mbox{supp}(f)\subset\bar{B}_{R}$. It is now obvious that $f^\#\in \mathcal{Y}_{M,R}$ by standard radially decreasing rearrangement properties. Therefore, $f^\#$ is a radial global minimizer of $E$ in $\mathcal{Y}_{M,R}$. This gives the desired result.
\end{proof}

We now give an equation satisfied by global minimizers of $E$ in
$\mathcal{Y}_{M,R}$. The following Lemma is essentially obtained by
\cite[Lemma 10]{Strohmer} and \cite{CCV}. In fact, global minimizers are solutions of
an obstacle problem.

\begin{lemma} \label{lem:statinary-solution}
Let $N \ge 3$ and $m > 2-\frac{2}{N}$, $M,R>0$. If $U_R \in \mathcal{Y}_{M,R}$ is a global minimizer of $E$ in $\mathcal{Y}_{M,R}$, then there exists a constant $\hat{C}$ such that $U_R$ satisfies
\begin{eqnarray}\label{mi-characterization}
\frac{m}{m-1} U_R^{m-1}(x)
& = & \Big(
V_R(x) + \hat{C} \Big)_+
\qquad
a.e. \ x \in B_R,
\end{eqnarray}
where $V_R=\Gamma\ast U_R\in L_{loc}^{\frac{m}{m-1}}(\R^N)$. Equivalently, $U_R$ satisfies
$$
\left\{ \begin{array}{ccc}
\frac{m}{m-1} U_R^{m-1}(x) - V_R(x)  = \hat{C} & , & \mbox{ if } U_R(x)>0 \\[4mm]
\frac{m}{m-1} U_R^{m-1}(x) - V_R(x)  \geq \hat{C} & , & \mbox{otherwise}
\end{array}\right.\,.
$$
\end{lemma}

\begin{proof}
Let $U_R \in \mathcal{Y}_{M,R}$ a global minimizer of $E$ in $\mathcal{Y}_{M,R}$.
We choose a compactly supported continuous function $\varphi \in L^1\cap L^m(B_R)$ such that
\begin{eqnarray}\label{phi-class}
\displaystyle \int_{\R^N} \varphi(x) \ dx = 0
\quad
{\rm and}
\quad 2|\varphi(x)| \le U_R(x) \quad {\rm for \ a.e. \ } x \in \R^N.
\end{eqnarray}
For such $\varphi$,
it holds that $U_R+\tau \varphi \in \mathcal{Y}_{M,R}$ for all $-1 \le \tau \le 1$.
Indeed, we easily see that
\begin{eqnarray*}
U_R \in L^1 \cap L^m(\R^N)
\quad
{\rm and}
\quad
\int_{\R^N} \Big( U_R(x) + \tau \varphi(x)   \Big) \ dx
            \ = \ \int_{\R^N} U_R(x) \ dx \ = \ M.
\end{eqnarray*}
In addition, \eqref{phi-class} implies that $\varphi(x)=0$ for a.e. $x \in \R^N \backslash B_R$ and
\begin{eqnarray*}
U_R(x) + \tau \varphi(x)
& \ge & U_R(x) - |\varphi(x)|
\ \ge \ U_R(x) - \frac{1}{2} U_R(x) \ \ge \ 0 \quad {\rm for \ a.e.} \ x \in \R^N.
\end{eqnarray*}

Since $U_R$ is a global minimizer of $E$ with support in the compact ball $B_R$ with good integrability properties, then we have that $E(U_R+ \tau \varphi)$ is a differentiable function with respect to $\tau$, see for instance \cite{Strohmer,Bian-Liu,CCV}, and
\begin{equation}\label{E-derivative}
\frac{d}{d\tau}
E[U_R+ \tau \varphi]|_{\tau=0}
=  \int_{\R^N} \left[ \frac{m}{m-1}U_R^{m-1}(x) - (\Gamma\ast U_R)(x) \right] \varphi(x) \ dx
= 0
\end{equation}
for all $\varphi$ with the property (\ref{phi-class}).
Now, let us take $\varphi$ as
\begin{eqnarray*}
\varphi(x)
& = &
\frac{1}{2} \Big( \phi(x) - \frac{1}{M} \int_{\R^N} \phi(y) U_R(y) \ dy  \Big) U_R(x),
\quad x \in \R^N,
\end{eqnarray*}
where $\phi \in C_0(B_R)$ is a function with the property that
$|\phi(x)| \le \frac{1}{2}$ for all $x \in \R^N$.
It is easily seen that $\varphi$ satisfies (\ref{phi-class}).
Indeed, since $U_R$ has mass $M$, then $\varphi$ has zero average and
\begin{eqnarray*}
2|\varphi(x)|
& \le &
\Big(
|\phi(x)| + \frac{1}{M} \max_{y \in \R^N} |\phi(y)| \int_{\R^N} U_R(y) \ dy
\Big) U_R(x)
\ \le \ U_R(x)
\end{eqnarray*}
for all $x \in \R^N$.
Hence, we deduce by (\ref{E-derivative}) that
\begin{equation}\label{h0-expression}
\int_{\R^N} \left[ \frac{m}{m-1}U_R^{m-1}(x) - (\Gamma\ast U_R)(x) \right] \Big( \phi(x) - \frac{1}{M} \int_{\R^N} \phi(y) U_R(y) \ dy  \Big) U_R(x)\ dx
= 0\,.
\end{equation}
Therefore, by denoting $F(U_R)$ by
\begin{eqnarray}\label{def-FU0}
F(U_R) & = &
\frac{m}{m-1} U_R^{m-1} - V_R
\end{eqnarray}
with $V_R=\Gamma\ast U_R$,
(\ref{h0-expression}) is equivalent to
\begin{eqnarray*}
\int_{\R^N}
\Big( F(U_R(x)) - \hat{C}   \Big) U_R(x) \phi(x) \ dx
& = &
0
\end{eqnarray*}
for all $\phi \in C_0(B_R)$, where $\hat{C}$ is given by
\begin{eqnarray*}%\label{def-C-hat}
\hat{C} & := & \frac{1}{M}\int_{\R^N} F(U_R(x)) U_R(x) \ dx.
\end{eqnarray*}
This implies by (\ref{def-FU0}) that
\begin{eqnarray}\label{F-omega-1}
\frac{m}{m-1} U_R^{m-1}(x) & = &  V_R(x) \ + \ \hat{C}
\qquad  {\rm for \ a.e.} \  x \in \{z \in \R^N; U_R(z) > 0\}.
\end{eqnarray}

To prove (\ref{mi-characterization}), it remains to treat the points where
$\Omega := \{x \in B_R; U_R(x) =  0\}$.
Now, we introduce $\tilde{\varphi}$ by
\begin{eqnarray*}
\tilde{\varphi}(x)
& = & \tilde{\phi}(x) - \frac{U_R(x)}{M} \int_{\R^N} \tilde{\phi}(y) \ dy,
\end{eqnarray*}
where $\tilde{\phi} \in C_0(\R^N)$ is a function with the properties
that $\tilde{\phi}(x)\ge 0$ for $x \in \R^N$ and ${\rm supp} (\tilde{\phi}) \subset B_R$.
Then we find
\begin{eqnarray*}
U_R+\tau \tilde{\varphi} \in \mathcal{Y}_{M,R} \qquad {\rm for \ all} \ 0 \le \tau \le \tau_0:=\frac{M}{2\int_{\R^N} \tilde{\phi}(x) \ dx}.
\end{eqnarray*}
Indeed, it holds that since $U_R, \tilde{\varphi} \in L^1 \cap L^m(\R^N)$ with support in $B_R$, then $\tilde{\varphi}$ has zero average and
\begin{eqnarray*}
U_R(x) + \tau \tilde{\varphi}(x)
& \ge & U_R(x) - \frac{M}{2\int_{\R^N} \tilde{\phi}(x) \ dx}
\frac{U_R(x)}{M} \int_{\R^N} \tilde{\phi}(y) \ dy
\nonumber \\
& = & \frac{U_R(x)}{2}
\ \ge \ 0, \qquad 0 \le \tau \le \tau_0 \quad
{\rm for \ a.e.} \ x \in \R^N,
\end{eqnarray*}
using that $\tilde{\phi}(x) \ge 0$ and $\tau\leq \tau_0$.

Again, since $U_R$ is a global minimizer of $E$ with support in the compact ball $B_R$ with good integrability properties, then we have that $E(U_R+ \tau \tilde \varphi)$ is a differentiable function with respect to $\tau$, see for instance \cite{Strohmer,Bian-Liu,CCV}, and then
\begin{equation*}
\frac{d}{d\tau}
E[U_R+ \tau \tilde{\varphi}]|_{\tau=0^+} =
\int_{\R^N} \Big( F(U_R(x)) - \hat{C}   \Big) \tilde{\phi}(x) \ dx \geq 0\,,
\end{equation*}
for all $\tilde \phi$. This implies that
$F(U_R(x)) - \hat{C} \geq 0$ for a.e. $x \in B_R$,
which gives the desired condition
\begin{equation}\label{F-omega-2}
0=\frac{m}{m-1} U_R^{m-1}(x)
\geq
V_R(x) + \hat{C}
\qquad {\rm for \ a.e.} \ x \in \Omega.
\end{equation}
Combining (\ref{F-omega-1}) with (\ref{F-omega-2}),
we can express $U_R$ by
\begin{eqnarray*}%\label{U0e-2}
\frac{m}{m-1} U_R^{m-1}(x)
& = & \Big(V_R(x) + \hat{C} \Big)_+ \quad {\rm for \ a.e.} \ x \in B_R\,,
\end{eqnarray*}
that finishes the nonlinear equation satisfied by the minimizer in \eqref{mi-characterization}.
\end{proof}

We can now use regularity theory to improve the properties of the global minimizers.

\begin{lemma} \label{lem:statinary-solution2}
Let $N \ge 3$ and $m > 2-\frac{2}{N}$, $M,R>0$. If $U_R \in \mathcal{Y}_{M,R}$ is a radially decreasing global minimizer of $E$ in $\mathcal{Y}_{M,R}$ with support $\bar{B}_{R_o}$, $0<R_o\leq R$, then
\begin{eqnarray} \label{Linfty-U0}
U_R \in L^1 \cap L^{\infty}(\R^N),
\end{eqnarray}
and thus $V_R\in L^q(\R^N)$ with $\frac{N}{N-2}<q<\infty$. Moreover,
$U_R \in C(B_R)\cap C^1(B_{R_o})\cap W^{2,p}(B_{R_o})$, $U_R^{m-1} \in W^{1,\infty}(B_R)\cap W^{2,p}(B_{R_o})$ for all $1<p<\infty$, $V_R\in C^1(\R^N)$, and
\eqref{mi-characterization} holds for all $x \in B_R$.
\end{lemma}

\begin{proof}
By (\ref{mi-characterization}), it suffices to show that $V_R \in L_{loc}^{\infty}(\R^N)$ to infer (\ref{Linfty-U0}). With this aim, we will prove that $U_R \in L^p(B_R)$ for certain $p>\frac{N}{2}$. This yields that $V_R\in W^{2,p}(B_R) \subset L^{\infty}(B_R)$ for $p>\frac{N}{2}$ by Weyl's Lemma and Morrey's theorem. We distinguish several cases:
\vspace{2mm} \\
(i) \ {\bf $m>\frac{N}{2}$}: Nothing to prove since $U_R\in L^m (B_R)$.
\vspace{2mm} \\
(ii) {\bf $m=\frac{N}{2}$}: Due to (\ref{mi-characterization}), H\"{o}lder's inequality and the Hardy-Littlewood-Sobolev inequality imply that
\begin{eqnarray*}
\| U_R\|_{L^N(B_R)}^{m-1}
& = &
\| U_R^{m-1} \|_{L^{\frac{N}{m-1}}(B_R)}
\ \le \
\| V_R \|_{L^{\frac{N}{m-1}}(B_R)}
\ + \ \hat{C}|B_R|^{\frac{m-1}{N}}
\\
& \le & C \|U_R\|_{L^{\frac{N}{m+1}}(B_R)} \ + \ \hat{C}|B_R|^{\frac{m-1}{N}}
\\
& \le & C \|U_R\|_{\frac{N}{2}} |B_R|^{\frac{m-1}{N}} \ + \ \hat{C}|B_R|^{\frac{m-1}{N}}
\label{m-N/2} \\
& = & \bigl( C \|U_R\|_{m} + \hat{C} \bigr) |B_R|^{\frac{m-1}{N}}
\ < \ \infty
\end{eqnarray*}
with $C=C(N)$, which yields $U_R \in L^N(B_R)$.
\vspace{2mm} \\
(iii) {\bf $2-\frac{2}{N} < m < \frac{N}{2}$}: We proceed by a bootstrap argument. Noting that
\begin{eqnarray*}
\frac{Nm(m-1)}{N-2m} \ > \ m,
\end{eqnarray*}
we have by (\ref{mi-characterization}) and the Hardy-Littlewood-Sobolev inequality that
\begin{eqnarray}
\| U_R \|_{L^{\frac{Nm(m-1)}{N-2m}}(B_R)}^{m-1}
& = &
\| U_R^{m-1} \|_{L^{\frac{Nm}{N-2m}}(B_R)}
\ \le \
\| V_R \|_{L^{\frac{Nm}{N-2m}}(B_R)}
\ + \ \hat{C}|B_R|^{\frac{N-2m}{Nm}}
\nonumber \\
& \le & C \|U_R\|_m \ + \ \hat{C}|B_R|^{\frac{N-2m}{Nm}} \ < \ \infty,
\label{Lemma5.3-ite}
\end{eqnarray}
where $C=C(N,m)$. Let us denote $m_1=\frac{Nm(m-1)}{N-2m}$.
We note that $m_1>1$ due to $m > 2-\frac{2}{N}$.
If $m_1>\frac{N}{2}$, {\it i.e.,} that $m^2>\frac{N}{2}$,
then we obtain the desired result.
In the case of $m_1=\frac{N}{2}$, we have by (\ref{Lemma5.3-ite}) that
$\|U_R\|_{L^{\frac{N}{2}}(B_R)} =\|U_R\|_{L^{m_1}(B_R)} \ < \ \infty$,
which together with case (ii) yields $\|U_R\|_{L^N(B_R)}<\infty$, and thus
we obtain the desired result.

In the case of $1<m_1<\frac{N}{2}$,
we obtain that
\begin{eqnarray*}
\| U_R \|_{L^{\frac{Nm_1(m-1)}{N-2m_1}}(B_R)}^{m-1}
& = &
\| U_R^{m-1} \|_{L^{\frac{Nm_1}{N-2m_1}}(B_R)}
\ \le \
\| V_R \|_{L^{\frac{Nm_1}{N-2m_1}}(B_R)}
\ + \ \hat{C}|B_R|^{\frac{N-2m_1}{Nm_1}}
\nonumber \\
& \le & C \|U_R\|_{m_1} \ + \ \hat{C}|B_R|^{\frac{N-2m_1}{Nm_1}} \ < \ \infty,
\end{eqnarray*}
where $C=C(N,m)$.

Now we proceed by induction defining a sequence $\{m_i\}_{i=0}^{\infty}$ by
\begin{eqnarray} \label{def-mi}
m_0 = m, \quad m_{i+1} & = & \frac{Nm_i(m-1)}{N-2m_i},
\quad i\in\mathbb{N} \cup \{0\}.
\end{eqnarray}
The previous procedure shows that $U_R \in L^{m_i}(B_R)$ for all $i$.
We shall show that there exists $i_* \in \mathbb{N}$ such that $m_{i_*+1} \ge \frac{N}{2}$
and $m_i < \frac{N}{2}$ for all $0\leq i \leq i_*$. This finishes the proof for this last case, since if $m_{i_*+1} = \frac{N}{2}$, we get that $U_R \in L^N(B_R)$ by the argument in case (ii), otherwise we have $m_{i_*+1} > \frac{N}{2}$.

Assume the contrary, {\it i.e.,} assume that $m_i<\frac{N}{2}$ for all $i\in\mathbb{N}$.
Then it holds by (\ref{def-mi}) that
\begin{eqnarray}\label{mi-positive}
m_i>0  \qquad {\rm for \ all} \ i\in\mathbb{N}.
\end{eqnarray}
On the other hand, the recursive formula (\ref{def-mi}) leads to
\begin{eqnarray*}
\frac{1}{m_i} & = &
\left\{
\begin{array}{lll}
\displaystyle A + B \frac{1}{(m-1)^i}
& {\rm for} \ 2-\frac{2}{N} < m <2, m>2, \\[4mm]
\displaystyle \frac{1}{2}  - \frac{2i}{N}
& {\rm for} \ m=2,
\end{array}
\right.\,
\end{eqnarray*}
for all $i\in\mathbb{N}$ with
$$
A=-\frac{2}{(m-2)N} \qquad \mbox{and} \qquad B=\frac1m+\frac{2}{(m-2)N}\,.
$$
Note that for $m>2$, $A$ is negative while the ${(m-1)^{-i}}\to 0$ as $i\to \infty$. Therefore, $m_i$ will become negative for large enough $i$ contradicting (\ref{mi-positive}).
Similarly, if $2-\frac{2}{N} < m <2$, then $A>0$, $B<0$, and ${(m-1)^{-i}}\to\infty$. Hence, we find that $m_i$ will be negative for large enough $i$ leading again to a contradiction with (\ref{mi-positive}). The case $m=2$ is obvious by the recursive formula above. This completes the proof of the first part of Lemma \ref{lem:statinary-solution2} since the regularity of $V_R$ is a direct consequence of the Hardy-Littlewood-Sobolev inequality.

Now, since $U_R \in L^1\cap L^\infty(\R^N)$, we deduce that $V_R \in L^{p}(\R^N)\cap W_{loc}^{2,p}(\R^N)$ for all $\frac{N}{N-2}<p<\infty$. In particular, it holds that $V_R \in C^1(\R^N)$.
By (\ref{mi-characterization}), we have $U_R \in C(B_R)$ and that (\ref{mi-characterization}) holds everywhere in $B_R$. Moreover,
$\frac{m}{m-1} U_R^{m-1}(x)= V_R(x) + \hat{C}$ for all $x \in B_{R_o}$ leading to $U_R^{m-1} \in W^{1,\infty}(B_R)\cap W^{2,p}(B_{R_o})$ for  all $1<p<\infty$. Observe also that $W^{2,p}(B_{R_o})\subset C^1(B_{R_o})$ for $p>N$, then $U_R\in C^1(B_{R_o})$.
This finishes the proof of the Lemma.
\end{proof}

%\begin{remark}
\noindent{\bf Remark.} \
Let us point out that $R_o<R$ or $R_o=R$ depending on the mass $M$. We also observe that all minimizers have the regularity stated in the previous Lemma in their supports taking into account that they are translations of radial global minimizers whose support is inside $B_R$.
%\end{remark}

\begin{cor}
\label{cor:stationary-sol-property}
Let $N \ge 3$ and $m > 2-\frac{2}{N}$, $M,R>0$.
Suppose that $U_R \in \mathcal{Y}_{M,R}$ is a radially decreasing global minimizer of $E$ in $\mathcal{Y}_{M,R}$ with support in $\bar{B}_{R_0}$, $R_o\leq R$. Then $U_R(x)=\phi(|x|)$ where $\phi\in C^1(I_o) \cap W^{2,p}(I_o) \cap C^0([0,R])$ for all $1<p<\infty$, with $I_o=(0,R_o)$, and $\phi^{m-1}$ is a solution of the  Lane-Emden equation \eqref{le} in Lemma {\rm\ref{lemma:local-existence}} for some $\alpha>0$ satisfying the properties \eqref{monotone-dec-phi}, \eqref{asymp-phi}, and \eqref{cmp-sp}. Moreover,  $R_o=\min(R_*(\alpha),R)$ with $R_*(\alpha)$ defined in Corollary {\rm\ref{cor:M-alpha}}.
\end{cor}

\begin{proof}
Since $U_R$ is spherically symmetric and with the regularity stated in Lemma \ref{lem:statinary-solution2}, i.e., $U_R \in C(B_R)\cap C^1(B_{R_o})\cap W^{2,p}(B_{R_o})$. Then there exists a function $\phi\in C^1(I_o) \cap W^{2,p}(I_o) \cap C^0([0,R])$ for all $1<p<\infty$, such that $U_R(x)=\phi(|x|)$. Moreover $\phi'(0^+)=0$ and $\phi^{m-1}(0)=\alpha$ where $\alpha^{\frac1{m-1}}$ is the maximum value of $U_R$.
Moreover, since $U_R^{m-1} \in W^{2,p}(B_{R_o})$ for all $1<p<\infty$, then we can take the Laplacian in $B_{R_o}$ in the identity \eqref{mi-characterization} to conclude
$$
\frac{m}{m-1}\Delta U_R^{m-1} = U_R
$$
a.e. in $B_{R_o}$. Therefore, $\psi=\phi^{m-1}$ satisfies
\begin{eqnarray}\label{LME}
\quad \psi^{\prime \prime}(r) + \frac{N-1}{r} \psi^{\prime}(r)
 =
-\frac{m-1}{m} \psi^{\frac{1}{m-1}}(r)\,, \mbox{ a.e. in } \R^N \,,
\end{eqnarray}
for all $0<r<R_o$. Thus, $\psi\in C^2(I_o)$ since the right-hand side of \eqref{LME} is continuous and $\psi$ is a solution of the Lane-Emden equation \eqref{le} for some $\alpha>0$ as stated. The properties of the solutions of the Lane-Emden equation in Lemma \ref{lemma:local-existence} imply the rest of the results.
\end{proof}

Finally, we can characterize all global minimizers in $\mathcal{Y}_{M,R}$ putting together Corollary
\ref{cor:stationary-sol-property} and Lemma \ref{lemma:ex-minimizer}.

\begin{cor}
\label{cor:stationary-sol-property2}
Let $N \ge 3$ and $m > 2-\frac{2}{N}$, $M,R>0$. Given $\tilde{U}_R \in \mathcal{Y}_{M,R}$ a global minimizer of $E$ in $\mathcal{Y}_{M,R}$. Then there exists $U_R(x)=\phi(|x|)$, $\phi$ with support in $I_o=(0,R_o)$, and with the properties in Corollary {\rm\ref{cor:stationary-sol-property}}, and there exists $y\in \R^N$ with $R_o+|y|\leq R$ such that $\tilde{U}_R(x)=\phi(|x+y|)$.
\end{cor}

%%%%%%%%%%%%%%%%%%%%%%%%%%%%%%%%%%%%%%%%%%%%

\subsection{Global minimizers in $\mathcal{Y}_{M}$ and Stationary solutions: Proof of Theorem \ref{thm:ch-statinary-solution}}
\label{section:reg-stationary-point}

We first observe that using Lemma \ref{lem:statinary-solution2} it is not difficult to check that the radial global minimizer found in Corollary \ref{cor:stationary-sol-property2} is a steady state of (KS) in the sense of Definition \ref{stationarystates} if its support lies inside $B_R$. The questions now are if we can show that this always happens and if we can show that all stationary states of (KS) are given by the global minimizer except translations. With this aim, we will make use of a powerful recent result concerning radial symmetry of stationary solutions in the sense of Definition \ref{stationarystates}. One of the main theorems in \cite{CHVY} applied to our particular problem implies the following:

\begin{thm}
\label{thm:unique} {\rm\cite[Theorem 2.2]{CHVY}}
Let $\rho_s \in L^1_+(\mathbb{R}^N)\cap L^\infty(\mathbb{R}^N)$ be a non-negative stationary state of {\rm (KS)} in the sense of Definition {\rm \ref{stationarystates}}. Then $\rho_s$ must be radially decreasing up to a translation, i.e. there exists some $x_0\in \mathbb{R}^d$, such that $\rho_s(\cdot - x_0)$ is radially symmetric, and $\rho_s(|x-x_0|)$ is non-increasing in $|x-x_0|$.
\end{thm}

Now, we are ready to show the main result of this paper.

\begin{thm} \label{thm:ch-statinary-solution2}
Let $N \ge 3$, $m > 2-\frac{2}{N}$, and $M>0$. There is a unique $\alpha=\alpha(M)$, such that the function $(U_M,V_M)=(U_{R_M},V_{R_M})$ with $R_M=R_*(\alpha(M))$ obtained in Corollary {\rm\ref{cor:stationary-sol-property}} verifies all the properties stated in Theorem \ref{thm:ch-statinary-solution}.
\end{thm}

\begin{proof}
Let us fix $M>0$ and take $U_R$ a radial global minimizer in $\mathcal{Y}_{M,R}$. Using Corollary \ref{cor:stationary-sol-property}, we deduce that there exists a function $\phi\in C^1(I_o) \cap W^{2,p}(I_o) \cap C^0([0,R])$ for all $1<p<\infty$, such that $U_R(x)=\phi(|x|)$ and $\psi=\phi^{m-1}$ satisfies the Lane-Emden equation
\begin{eqnarray*}%\label{LME}
\quad \psi^{\prime \prime}(r) + \frac{N-1}{r} \psi^{\prime}(r)
 =
-\frac{m-1}{m} \psi^{\frac{1}{m-1}}(r)\,,
\end{eqnarray*}
with $\psi'(0^+)=0$ and $\psi(0)=\alpha$ where $\alpha^{\frac1{m-1}}$ is the maximum value of $U_R$. Proceeding as in Corollary \ref{cor:M-alpha} and in \cite[Lemma 13]{S-Z}, given the solution $\Psi$ of the Lane-Emden equation \eqref{le} with the initial data
$\Psi(0)=1$, one can easily check that $\psi$ defined as
$$
\psi(r) = \alpha \Psi(\alpha^\mu r), \qquad \mbox{with }
\mu=\frac{1}{2}\left(\frac{2-m}{m-1}\right)\,,
$$
for every $\alpha >0$ is the solution of \eqref{le} with
$\psi(0)=\alpha$. This readily implies $R_*(\alpha)=\alpha^{-\mu}
R_*(1)$. For $\alpha>0$, we define $M_2(\alpha,R)$ for $R\leq R_*(\alpha)$ by
\begin{equation}\label{m2}
M_2(\alpha,R) := \omega_N \int_0^{R} \left[\alpha
\Psi(\alpha^\mu r)\right]^{\frac{1}{m-1}} r^{N-1} \ dr= \omega_N
\alpha^{\frac{1}{m-1}-\mu N} \int_0^{R\alpha^\mu}
\Psi(s)^{\frac{1}{m-1}} s^{N-1} \ ds
\end{equation}
and $M_2(\alpha,R)=M_2(\alpha,R_*(\alpha))$ for $R>R_*(\alpha)$. We clearly observe that $M_2(\alpha,R)$ is strictly increasing in $0<R<R_*(\alpha)$ and then constantly equal to $M_1(\alpha)$ for $R\geq R_*(\alpha)$ and $\alpha$ fixed, see Corollary \ref{cor:M-alpha} for the definition of $M_1(\alpha)$. We also obtained in Corollary \ref{cor:M-alpha} that $M_2(\alpha,R)$ is strictly increasing as a function of $\alpha$ for $R>R_*(\alpha)$. We remind the reader that in the diffusion-dominated case $\frac{1}{m-1}-\mu N>0$. Let us distinguish two cases:
\begin{itemize}
\item If $2-\frac2N<m\leq 2$, then $\mu\geq 0$ and the function $M_2(\alpha,R)$ is strictly increasing in $\alpha$ for fixed $R$ just looking at the second expression in \eqref{m2} with $M_2(\alpha,R)\to \infty$ as $\alpha\to \infty$.

\item If $m> 2$, then $\mu< 0$ and the function $M_2(\alpha,R)$ is strictly increasing in $\alpha$ using the first expression of $M_2(\alpha,R)$ in \eqref{m2}, due to the fact that the function $\Psi(\alpha^\mu r)$ is strictly increasing in $\alpha$. For the same reason, $M_2(\alpha,R)$ diverges too as $\alpha\to \infty$.
\end{itemize}
Therefore, for each $R>0$, there is a unique $\alpha(M,R)$ such that $M_2(\alpha,R)=M$, and thus there is a unique radial minimizer $U_R$ in $\mathcal{Y}_{M,R}$ for each $R>0$. Moreover, since $M_2(\alpha,R)$ is constantly equal to $M_1(\alpha)$ for $R\geq R_*(\alpha)$, then $\alpha(M,R)$ is constant for $R$ large enough and equal to the unique $\alpha(M)$ such that $M_1(\alpha)=M$. This shows that $U_R$ does not depend on $R$ if $R$ is large enough. Therefore, this together with \eqref{limmins} allow us top conclude that there is a unique radial global minimizer in $\mathcal{Y}_{M}$. The rest of global minimizers are obtained by translation in view of Corollary \ref{cor:stationary-sol-property2}. This finishes the proof of the first statement (i).

The statement (ii) is almost a direct consequence of the regularity properties in Lemma {\rm\ref{lem:statinary-solution2}} for $U$ and $V$. The remaining property is the classical regularity of $V$. Since $U^{m-1}\in W^{1,\infty}(\R^N)$, then $U \in C_0 \cap C^{0,\frac{1}{m-1}}(\R^N)$ for $m\geq 2$ and $U \in C^1_0(\R^N)$ for $2-\frac2N<m\leq 2$,
and hence the newtonian potential $V \in C^2(\R^N)$ by standard classical regularity theory. Then $U$ is a stationary solution of the (KS) system in the sense of Definition \ref{stationarystates} and in the classical sense in its support. We remind the reader that we already checked the first equation in the support of $U$ due to the nonlinear nonlocal equation \eqref{mi-characterization}.

To show (iii), we follow a similar argument employed in Aronson \cite{Ar3}.
Let $U(x_0)>0$. Then we see by (ii) that $\nabla U^{m-1+\delta}$ with $\delta>0$
are continuous functions in a neighbourhood of $x_0$. It is also obviously true outside the
support of $U$. Therefore, it suffices to prove that
$\nabla U^{m-1+\delta}$ is a continuous function in a
neighbourhood of $x_1\in \partial B_{R_M}$ with the additional property that $\nabla U^{m-1+\delta}(x_1)=0$.
Indeed, since $U \in C_0(\R^N)$, for every $\varepsilon$, there exists $a>0$ such that
\begin{eqnarray} \label{hani}
0 \ \le \ U(x)
& \le &
|U(x)-U(x_1)| + U(x_1)
\ \le \  \varepsilon
\end{eqnarray}
holds for all $x \in I_a(x_1):=\{x \in \overline{B_{R_M}}; |x-x_1| < a\}$. On the other hand, since we have
\begin{eqnarray}\label{sekibun}
\lefteqn{
|U^{m-1+\delta}(x)-U^{m-1+\delta}(x^{\prime})|
} \nonumber \\
& \le &
|x-x^{\prime}|
         \int_0^1 |(\nabla U^{m-1+\delta})(\tau x + (1-\tau) x^{\prime})| \ d\tau
\nonumber \\
& = & |x-x^{\prime}|  \frac{m-1+\delta}{m-1}
         \int_0^1 |U^{\delta}(\tau x + (1-\tau) x^{\prime})
                            \nabla U^{m-1}(\tau x + (1-\tau) x^{\prime})| \ d\tau
\nonumber \\
& \le & |x-x^{\prime}|  \frac{m-1+\delta}{m-1}
         \|U\|_{\infty}^{\delta}
         \|\nabla U^{m-1}\|_{\infty},
\end{eqnarray}
for all $x,x^{\prime} \in I_a(x_1)$. It follows from (\ref{hani}) and (\ref{sekibun}) that
\begin{eqnarray}\label{ilet}
|U^{m-1+\delta}(x)-U^{m-1+\delta}(x^{\prime})|
& \le &
C \varepsilon^{\delta} |x-x^{\prime}|
\qquad
{\rm for \ all} \ x,x^{\prime} \in I_a(x_1)
\end{eqnarray}
and for all  $0 < \varepsilon \le 1$, where $C=C(N,m,\delta)$.
Taking $x=x_1$ in (\ref{ilet}) and then letting $x^{\prime} \to x_1$, we have
$|\nabla U^{m-1+\delta}(x_1)| \le C\varepsilon^{\delta}$, $0<\varepsilon \le 1$.
Hence we have by letting $\varepsilon \to 0$ that
$\nabla U^{m-1+\delta}(x_1)=0$. Similarly, letting $x^{\prime} \to x$ in (\ref{ilet}), we have
\begin{eqnarray*}\label{base-2}
|\nabla U^{m-1+\delta}(x)| & \le & C\varepsilon^{\delta}
\qquad {\rm for \ all} \ 0<\varepsilon\le 1,
\end{eqnarray*}
which implies that $\nabla U^{m-1+\delta}$ is continuous at $x_1$.
We conclude that
$\nabla U^{m-1+\delta}$ is a continuous function in $\R^N$
with the additional property that
$\nabla U^{m-1+\delta}(x)=0$ at the boundary of the support.

The statement (iv) is just to collect all properties in one statement as it is a direct consequence of \eqref{mi-characterization} and (i).

To finish the proof, we only need to show the uniqueness statement of Theorem \ref{thm:ch-statinary-solution}. Assume that $\tilde{U}$ is another stationary solution of {\rm (KS)} in the sense of Definition {\rm\ref{stationarystates}} then $\tilde U \in L^1\cap L^\infty(\R^N)$ from which we deduce that $\tilde V \in L^{p}(\R^N)\cap W^{1,\infty}(\R^N)\cap W_{loc}^{2,p}(\R^N)$ for all $\frac{N}{N-2}<p<\infty$. In particular, it holds that $\tilde V \in C^1(\R^N)$. Following the arguments in \cite[Lemma 2.3]{CHVY}, see also \cite{Strohmer,CCV}, one can prove from \eqref{steady} that $\tilde U \in C^0(\R^N)$
and
\[
\frac{m}{m-1} \tilde U^{m-1}(x) + \tilde V (x) \quad \text{ is constant in each connected component of }\mbox{supp } \tilde U.
\]
Due to the regularity of the stationary solution and the main result in \cite[Theorem 2.2]{CHVY} applied to our particular case in Theorem \ref{thm:unique}, we deduce that the density $\tilde U$ is a radially decreasing continuous function except translations. Let us assume without loss of generality that the center of mass is zero. Moreover, this fact implies that $\mbox{supp } \tilde U$ has a single connected component and $\mbox{supp } \tilde U=\bar B_{\tilde R}$ for some $0<\tilde R \leq \infty$.

Since there exists a constant $\tilde C$ such that $\frac{m}{m-1} \tilde U^{m-1}(x)= \tilde V(x) + \tilde{C}$ for all $x \in B_{\tilde R}$ then $\tilde U^{m-1} \in W^{1,\infty}(\R^N)\cap W^{2,p}_{loc}(B_{\tilde R})$ for  all $1<p<\infty$, and Sobolev emmbeddings imply that $\tilde U\in C^1(B_{\tilde R})$. Summarizing, $\tilde U$ has the same regularity and properties of the radial global minimizer as in Corollary \ref{cor:stationary-sol-property}. Therefore, there exists a function $\tilde \phi\in C^2((0,\tilde R)) \cap W^{2,p}((0,\tilde R) \cap C^0([0,\tilde R])$ for all $1<p<\infty$, such that $\tilde U(x)=\tilde\phi(|x|)$ solution to the Lane-Emden equation \eqref{le} for some $\tilde \phi^{m-1}(0)=\tilde \alpha$. Since the stationary solution $\tilde U$ has mass $M$ and using the first statement of this theorem, there is a unique solution of the Lane-Emden equation with mass $M$ that is given by the global minimizer $\phi$, and thus $\tilde \phi=\phi$, implying in particular that $\tilde R<\infty$ and $\tilde U$ is compactly supported. Now, the proof of Theorem \ref{thm:ch-statinary-solution} is complete.
\end{proof}

\subsection*{Acknowledgments}
JAC was supported by projects MTM2011-27739-C04-02 and from
the Royal Society through a Wolfson Research Merit Award. YS was supported by the Japan Science and Technology Agency (JST), PRESTO. The authors are very grateful to the Mittag-Leffler Institute for providing a fruitful working environment during the special semester \emph{Interactions between Partial Differential Equations \& Functional Inequalities}.

\bibliographystyle{siam}\small
\bibliography{biblio}

\end{document}